\documentclass[12pt,epsfig,amsfonts]{amsart} 
\usepackage{amsmath,amsthm,amssymb,amscd,epsfig,color}
\usepackage{graphicx}
\usepackage{subfigure}
\usepackage{subfig}
\usepackage{mathrsfs}
\usepackage{comment}
\usepackage{mathrsfs}

\makeatletter
\newcommand{\tpitchfork}{%
  \vbox{
    \baselineskip\z@skip
    \lineskip-.52ex
    \lineskiplimit\maxdimen
    \m@th
    \ialign{##\crcr\hidewidth\smash{$-$}\hidewidth\crcr$\pitchfork$\crcr}
  }%
}
\makeatother

\newtheorem*{theorema}{Theorem A}
\newtheorem*{theoremb}{Theorem B}
\newtheorem*{theoremA4.1}{the Main Theorem3.1}
\newtheorem*{theoremc}{Theorem~C}

\newtheorem{prop}{Proposition}[section]
\newtheorem{lemma}[prop]{Lemma}

\newtheorem{theorem}[prop]{Theorem}

\numberwithin{equation}{section}

\author{Yoshitaka Saiki, Hiroki Takahasi, and James A. Yorke}
\address{Graduate School of Business Administration, Hitotsubashi University, Tokyo,
186-8601, JAPAN} 
\email{yoshi.saiki@r.hit-u.ac.jp}
\address{Keio Institute of Pure and Applied Sciences (KiPAS), Department of Mathematics,
Keio University, Yokohama,
223-8522, JAPAN} 
\email{hiroki@math.keio.ac.jp}

\address{Institute for Physical Science and Technology and Mathematics and Physics Departments,
University of Maryland College Park, MD 20742, USA} 
\email{yorke@umd.edu}

\begin{document}
\title[Hausdorff dimension of Cantor intersections]
{Hausdorff dimension of Cantor 
intersections\\
and robust heterodimensional cycles for\\
heterochaos horseshoe maps}
\maketitle

   \begin{abstract}
   As a model to provide a hands-on, elementary understanding of chaotic dynamics in dimension three,
we introduce a $C^2$-open set of diffeomorphisms of 
  $\mathbb R^3$ having two horseshoes with different dimensions of instability. We prove that: 
 the unstable set of one horseshoe and the stable set of the other are of Hausdorff dimension nearly $2$ 
  whose cross sections are Cantor sets;
  the intersection of the unstable and stable sets contains a fractal set of Hausdorff dimension nearly $1$. As a corollary we detect $C^2$-robust heterodimensional cycles.
  Our proof employs
  the theory of normally hyperbolic invariant manifolds
  and the thicknesses of Cantor sets. 
  \end{abstract}

\section{Introduction}
The fractal theory of hyperbolic
sets (horseshoes) for surface diffeomorphisms is a well-developed topic, 
 see for instance \cite{McMan83,PalVia88,Tak86}. Fractal quantities such as the Hausdorff dimension and limit capacity,
have been brought into the bifurcation theory of surface diffeomorphisms and played an important role. 
Newhouse \cite{New70} defined a non-negative quantity called
the ``thickness'' of a Cantor set on the real line, in order to formulate conditions which guarantee that two Cantor sets intersect each other.
These conditions have been applied to 
surface diffeomorphisms 
 to show the robustness of tangencies between unstable and stable manifolds
 whose cross sections are Cantor sets \cite{New70,New74,New79,PalTak93,Rob83}.
For extensions of these results to higher dimensions, see
\cite{GTS93,PalVia,Rob83,Rom95}.

A fundamental property of surface diffeomorphisms is that
all non-trivial hyperbolic sets have index one (the dimension of the unstable subbundle). 
In higher dimension,
the situation is rather more complicated when hyperbolic sets have different
indices and different horseshoes with different indices are  cyclically related. 
The question of getting sufficient conditions for robust dynamical phenomena,
including tangencies, heterodimensional cycles, is an active field of research.

The present paper aims to answer this
sort of question on the Hausdorff dimension of heteroclinic intersections for certain intermingled horseshoes. 
We provide an elementary example of diffeomorphisms in dimension three that display a new type of robust fractal non-transverse intersection
 between unstable and stable manifolds. This example is a variation of
 the piecewise affine map $F$ on the cube $[0,1]^3$ in the Euclidean space $\mathbb R^3=\{(x_u,x_c,x_s)\colon x_u,x_c,x_s\in\mathbb R\}$ introduced in \cite{STY20}, as a model to provide a hands-on, elementary understanding of complicated dynamics in dimension three.
 Define $h\colon[0,1]\to[0,1]$ by $h(x)=3x\mod1$ for $x\in[0,1)$ and $h(1)=1$.
\begin{figure}
    \centering
    \includegraphics[width=0.65\columnwidth]{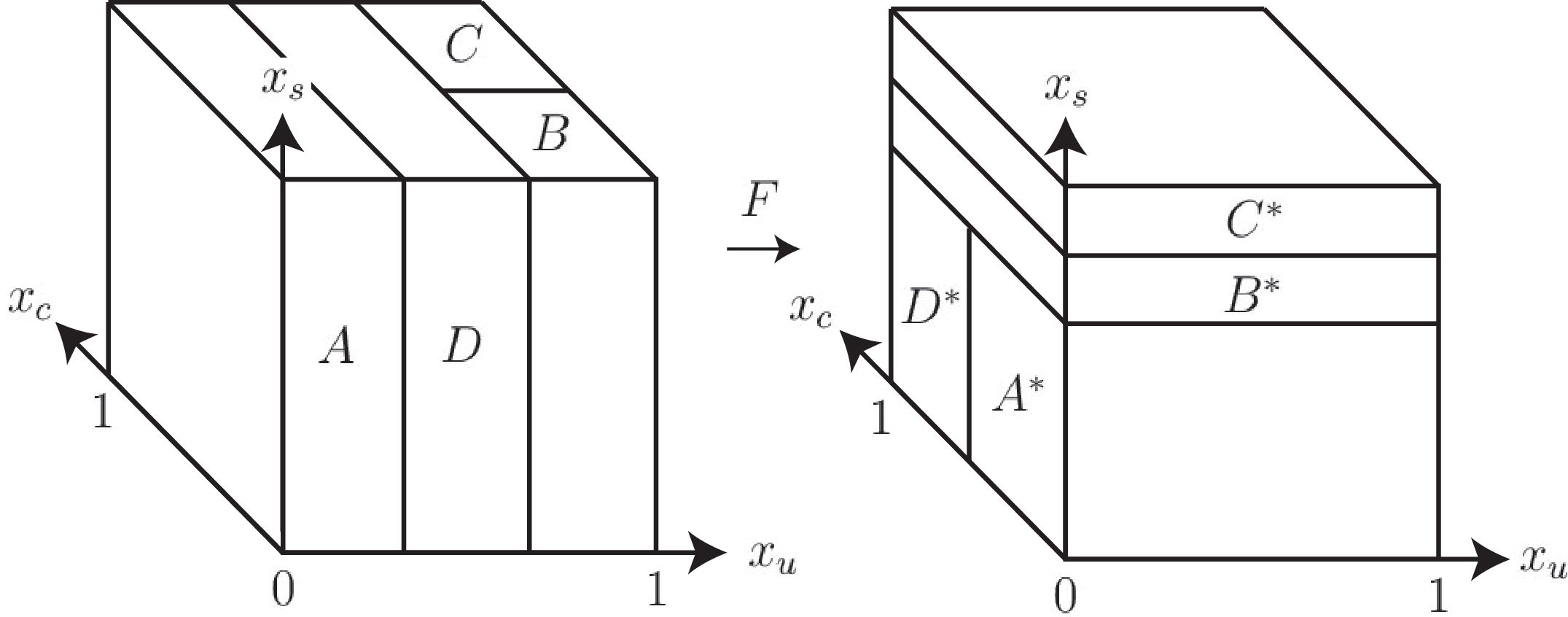}
    \caption{The piecewise affine map $F$. $A^*=F(A)$,
    $B^*=F(B)$, $C^*=F(C)$, $D^*=F(D)$.}
    \label{fig:heterochaos-baker3d}
\end{figure}
The map $F\colon[0,1]^3\to[0,1]^3$ is given by 
\[  F(x_u,x_c,x_s)=
  \begin{cases}
  \vspace{1mm}
    \displaystyle{\left(h(x_u),~~~~~\frac{1}{2}x_c,~~~~~~~\frac{2}{3}x_s\right)}&\text{on}~A, \\
    \vspace{1mm}
     \displaystyle{\left(h(x_u),2x_c,~~~~~~~\frac{1}{6}x_s+\frac{2}{3}\right)}&\text{on}~B, \\
     \vspace{1mm}
   \displaystyle{\left (h(x_u),2x_c-1,\frac{1}{6}x_s+\frac{5}{6}\right)}&\text{on}~C, \\
    \displaystyle{\left(h(x_u),~~~~~\frac{1}{2}x_c+\frac{1}{2},~~\frac{2}{3}x_s\right)}&\text{on}~D,
  \end{cases}\]
where 
\[\begin{split}
A =& \left[0,\frac{1}{3}\right)\times
\left[0,1\right]\times\left[0,1\right],\\
B =&\left [\frac{2}{3},1\right]\times
\left[0,\frac{1}{2}\right)\times\left[0,1\right],\\
C =& \left[\frac{2}{3},1\right]\times
\left[\frac{1}{2},1\right]\times\left[0,1\right],\\
D =&\left[\frac{1}{3},\frac{2}{3}\right)
\times\left[0,1\right]\times\left[0,1\right].
\end{split}\]
See FIGURE~\ref{fig:heterochaos-baker3d}.
 The map $F$ 
 preserves the Lebesgue measure on $[0,1]^3$, and has
a pair of closed invariant sets  
$\bigcap_{n=-\infty}^{\infty}F^n(A\cup D)$ and
$\bigcap_{n=-\infty}^{\infty}F^n(B\cup C)$ of indices $1$ and $2$ respectively. 
The intersection of the unstable set of the first set and the stable set of the second set contains the segment $\{(1,x_c,0)\colon 0\leq x_c\leq1\}$, which means that the Hausdorff dimension of their heteroclinic intersection is at least $1$.
The following dynamical properties of $F$ were proved in \cite{STY20}:
\begin{itemize}\item[$\circ$] $F$ is transitive, i.e., has a dense orbit in $[0,1]^3$.

\item[$\circ$] The set of hyperbolic periodic points with index $1$ is dense in $[0,1]^3$.

\item[$\circ$] The set of hyperbolic periodic points with index $2$ is dense in $[0,1]^3$.

\item[$\circ$] $F$ is weak mixing with respect to the Lebesgue measure.
\end{itemize}
In this paper we modify $F$ into a diffeomorphism, and
investigate the Hausdorff dimension of the intersection of the invariant manifolds of two intermingled horseshoes, allowing $C^2$ perturbations.
As a corollary we obtain a new type of $C^2$-robust heterodimensional cycles.

\subsection{Heterochaos horseshoe maps}\label{cou-sec}
We begin with basic definitions to introduce our diffeomorphisms.
A {\it block} is a Cartesian product of three non-degenerate compact  intervals in $\mathbb R^3$ all of whose sides are parallel to one of the axes of coordinates of $\mathbb R^3$.
Given two blocks $X$ and $Y$, we say $X$ {\it stretches across} $Y$
if $X$ does not intersect the edges of $Y$ and $X\setminus Y$ has two connected components.

\begin{figure}
\centering
\includegraphics[width=.9\textwidth]{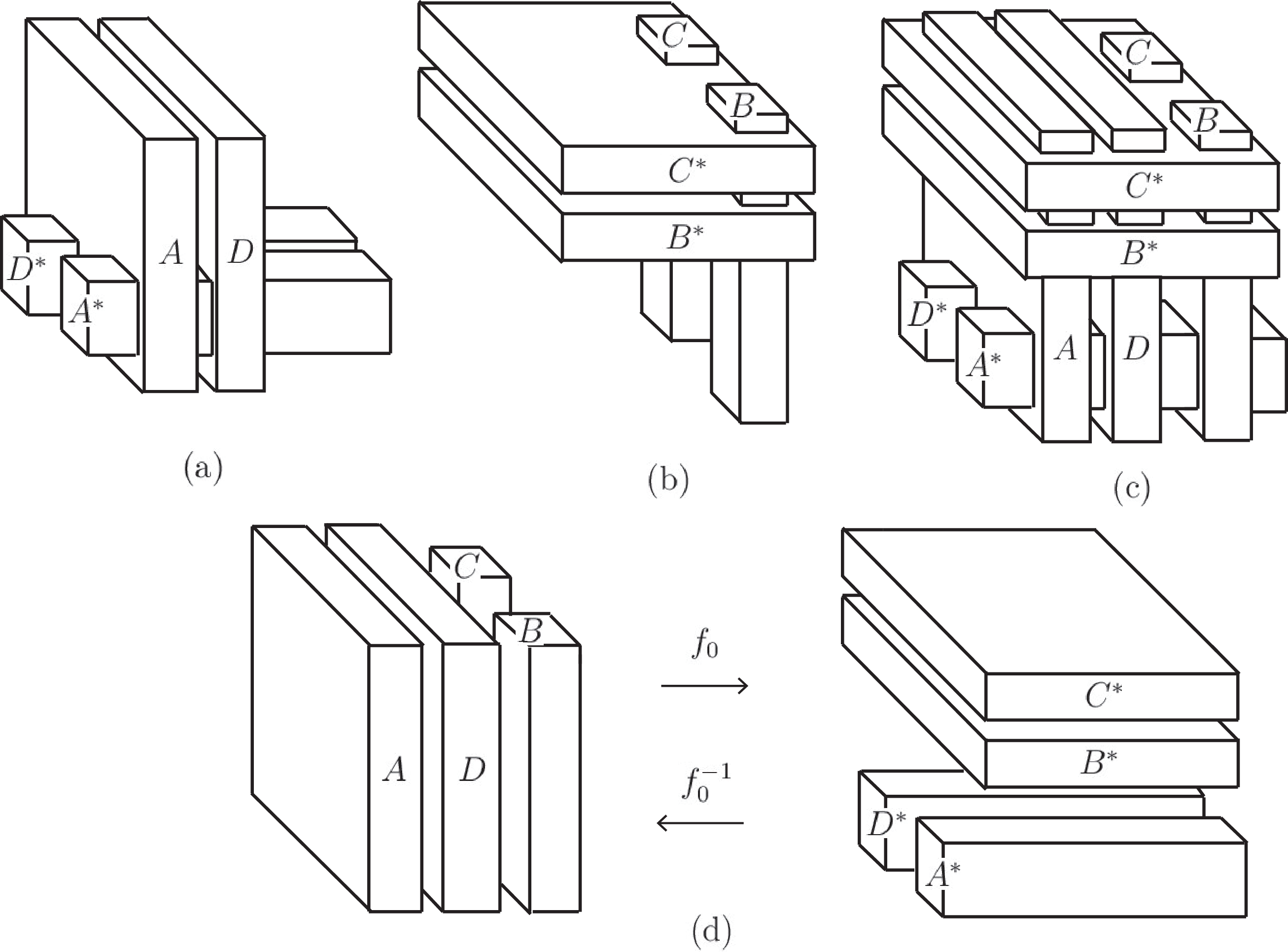}
\caption{The heterochaos horseshoe map $f_0$.}
\label{fig:coupledhorseshoe}
\end{figure}

 Let $A$, $B$, $C$, $D$ be pairwise disjoint blocks
in $[0,1]^3$ such that
 $D$ is a translate of $A$ in the $x_u$-direction and
 $C$ is a translate of $B$
 in the $x_c$-direction. Similarly,
 let $A^*$, $B^*$, $C^*$, $D^*$
be pairwise disjoint 
blocks in $[0,1]^3$
such that $D^*$ is a translate of $A^*$ in the $x_c$-direction and
 $C^*$ is a translate of $B^*$
 in the $x_s$-direction.
We assume:

\begin{itemize}
 \item[(i)] $A^*$  stretches across $A$ and $D$, and $D^*$ stretches across $A$ and $D$.

\item[(ii)] $B$ stretches across $B^*$ and $C^*$, and $C$ stretches across $B^*$ and $C^*$.

\item[(iii)] 
$A^*$ stretches across $B$, and $D^*$ stretches across $C$.

\item[(iv)] $A$ stretches across $B^*$ and $C^*$, and
$D$ stretches across $B^*$ and $C^*$.
\end{itemize}
\noindent See FIGURE~\ref{fig:coupledhorseshoe}(a), (b), (c), (c) respectively. 
Set \[R_1=A\cup D\quad\text{and}\quad R_2= B^*\cup C^*.\]

For an integer $r\geq1$,  
let ${\rm Diff}^r(\mathbb R^3)$ denote the space of $C^r$ diffeomorphisms of $\mathbb R^3$
endowed with the $C^r$ compact open topology.
We say $f_0\in{\rm Diff}^1(\mathbb R^3)$ is
a {\it heterochaos horseshoe map}
if it maps $A$, $B$, $C$, $D$ affinely to $A^*,B^*,C^*,D^*$ respectively
with diagonal Jacobian matrices, see FIGURE~\ref{fig:coupledhorseshoe}(d).
Condition
(i) implies that the restriction
$f_0|_{R_1}$ is a horseshoe map
whose unstable manifolds are one-dimensional. Condition
 (ii) implies that
$f_0|_{R_2}$
is a horseshoe map  whose stable manifolds are one-dimensional.
Conditions (iii) (iv) determine how these two horseshoe maps are coupled.

\subsection{Statements of results}
Let $f_0$ be a 
heterochaos horseshoe map. Let $f\in{\rm Diff}^1(\mathbb R^3)$
be sufficiently $C^1$-close to $f_0$.
The compact $f$-invariant set
\[\Lambda=\bigcap_{n=-\infty}^{\infty} f^{-n}(R_1\cup R_2)\]
contains two closed $f$-invariant sets
$\Gamma=\Gamma(f)$ and $\Sigma=\Sigma(f)$ given by
\[\Gamma=\bigcap_{n=-\infty}^{\infty} f^{-n}(R_1)\quad\text{and}\quad 
\Sigma=\bigcap_{n=-\infty}^{\infty} f^{-n}(R_2),\]
which are hyperbolic sets of indices $1$ and $2$ respectively.
Extending the arguments in \cite{STY20} it is possible to show the transitivity, as well as the density of hyperbolic periodic points with different indices in $\Lambda$.
 If $f$ is $C^2$, the Lebesgue measure of $\Lambda$ is zero.

Let $W^u(\Gamma)$ (resp. $W^s(\Sigma)$)
denote the union of the unstable (resp. stable) manifolds of points in $\Gamma$ (resp. $\Sigma)$.

We will provide a condition on $f_0$ which ensures that
 the set 
\[\mathscr{H}=W^u(\Gamma)\cap W^s(\Sigma)\cap\Lambda\]
is non-empty for $f$ which is sufficiently $C^2$-close to $f_0$. Under a stronger condition on $f_0$, we will give an estimate of the Hausdorff dimension of
$\mathscr{H}$.

To these ends, we will define thicknesses of cross sections of $W^u(\Gamma)$
 and $W^s(\Sigma)$, and relate them to 
the following numbers
\[a_1=\frac{|A_c^*|-\kappa_1|A_c^*| }{|E_c^*|-2|A_c^*|+\kappa_1|A_c^*|
 }\quad\text{and}\quad a_2=\frac{|B_c|-\kappa_2|B_c| }{|F_c|-2|B_c|+\kappa_2|B_c|},\]
where
\[\kappa_1=\frac{|A_c|-|E_c^*|}{|A_c|-|A_c^*|}\quad\text{and}\quad \kappa_2=\frac{|B_c^*|-|F_c|}{|B_c^*|-|B_c|}.\]
Hereafter,
 $E^*$ denotes the minimal block containing $A^*\cup D^*$, 
 and $F$ denotes the minimal block containing $B\cup C$.
 For a block $X$
we write $X=X_u\times X_c\times X_s$.
We denote by $|I|$ the Euclidean length of a bounded interval $I\subset\mathbb R$.
Note that $\kappa_1,\kappa_2\in(0,1)$, and so $a_1,a_2>0$.

Let $\dim_{\rm H}$ denote the Hausdorff dimension on $\mathbb R^3$. 
The key term {\it dimension-reducible} will be defined in Section~\ref{constants-s}.
Our main results are stated as follows.

\begin{theorema}
\label{cycle-thm} 
Let $f_0\in{\rm Diff}^2(\mathbb R^3)$ be a heterochaos horseshoe map that is dimension-reducible 
and satisfies $a_1a_2>1$.
There exists a $C^2$ neighborhood $\mathcal V$ of $f_0$  such that
for any $f\in\mathcal V$, $\mathscr{H}$ is a non-empty totally disconnected set satisfying
\[\dim_{\rm H} \mathscr{H}<1.\]
Moreover we have
\[\begin{split}&
\frac{\log2}{\log\left(2+2/a_1\right)}+1<\dim_{\rm H} W^u(\Gamma)< 2,\quad
\frac{\log2}{\log\left(2+2/a_2\right)}+1<\dim_{\rm H} W^s(\Sigma)< 2.\end{split}\]
\end{theorema}
 
  The sets $W^u(\Gamma)$ and $W^s(\Sigma)$ are contained in locally invariant $C^2$ surfaces.
We believe it is possible to characterize their Hausdorff dimensions
as zeros of appropriately defined pressure functions as in \cite{McMan83}.
For this type of results in dimension three, we refer the reader to \cite{DGGJ19,SS99}. 
  
  The sets $W^u(\Gamma)$ and $W^s(\Sigma)$
 intersect each other as depicted in FIGURE~\ref{fig:heteroclinic-intersection}. All their intersections are quasi-transverse, while
  $W^s(\Gamma)$ intersects $W^u(\Sigma)$  transversely. Therefore, the two hyperbolic sets $\Gamma$ and $\Sigma$ are cyclically related. 
 This situation leads us to 
 state an immediate dynamical application of Theorem~A.
 We say $f\in{\rm Diff}^r(\mathbb R^3)$ $(r\geq1)$ has a {\it heterodimensional cycle associated to its transitive hyperbolic sets $\Psi$ and $\Upsilon$} if the indices of  $\Psi$ and $\Upsilon$ are different, and
 $W^s(\Psi)\cap W^u(\Upsilon)\neq\emptyset$ and $W^s(\Upsilon)\cap W^u(\Psi)\neq\emptyset$.
 Since non-transverse intersections in heterodimensional cycles
 can easily be destroyed by perturbations of diffeomorphisms, the following notion is useful \cite{BonDia08}.
A heterodimensional cycle of $f$ associated to transitive hyperbolic sets $\Psi$ and $\Upsilon$ is {\it $C^r$-robust}
if there is a $C^r$ neighborhood $\mathcal V$ of $f$
such that any diffeomorphism in $\mathcal V$
has a heterodimensional cycle associated to the 
continuations of $\Psi$ and $\Upsilon$.
By Kupka-Smale's theorem, if $f$ has a $C^r$-robust heterodimensional cycle associated to $\Psi$ and $\Upsilon$, either $\Psi$ or $\Upsilon$ is a non-trivial hyperbolic set.

Robust heterodimensional cycles are often constructed by perturbations. It may be useful to exhibit a concrete example of a diffeomorphism
which has a robust heterodimensional cycle. As a corollary to Theorem~A we obtain the following.

 \begin{theoremb}\label{hetero-cor}
 Let $f_0\in{\rm Diff}^2(\mathbb R^3)$ be a heterochaos horseshoe map that is dimension-reducible 
and satisfies $a_1a_2>1$.
  Then $f_0$ has a $C^2$-robust heterodimensional cycle
 associated to $\Gamma$ and $\Sigma$.
\end{theoremb}

The $C^2$-robust intersection between $W^u(\Gamma)$ and $W^s(\Sigma)$ obtained in Theorem~A, as well as the $C^2$-robust heterodimensional cycle in Theorem~B
comes from the analysis of the intersection of a pair of Cantor sets on the real line. 
In all our results, including Theorem~C below, the thicknesses of such Cantor sets are essentially used, and therefore
it is not possible to weaken the $C^2$ topology to $C^1$. 
Indeed, Ures \cite{Ure95} showed that the thicknesses of $C^1$-generic regular Cantor sets are zero, and
 Moreira \cite{Mor11} showed that  
 any two intersecting regular Cantor sets can be $C^1$-perturbed so that the perturbed Cantor sets do not intersect each other.

From Theorem~A, we observe that 
the Hausdorff dimension of $W^u(\Gamma)$ and that of $W^s(\Sigma)$ converge to $2$ as $\min\{a_1,a_2\}\to\infty$.
Our next result shows that $\mathscr{H}$ does not contain a continuum, but contains a set of Hausdorff dimension nearly $1$.

\begin{theoremc}
\label{main}
 There exists $T_0>1$ 
 such that if $f_0\in{\rm Diff}^2(\mathbb R^3)$ is a
 heterochaos horseshoe map that
  is  dimension-reducible and satisfies 
 $\min\left\{a_1,
 a_2\right\}>T_0$, 
then there exists a $C^2$ neighborhood $\mathcal V$ of $f_0$ such that
for any $f\in\mathcal V$,
\[\begin{split}
\left(1+ \frac{1}{\sqrt{\min\{a_1,a_2\}}}\right)^{-1}<
\dim_{\rm H} \mathscr{H}<1.\end{split}\]
 \end{theoremc}
 
\begin{figure}
\centering
\includegraphics[height=0.5\linewidth,width=0.5\linewidth]{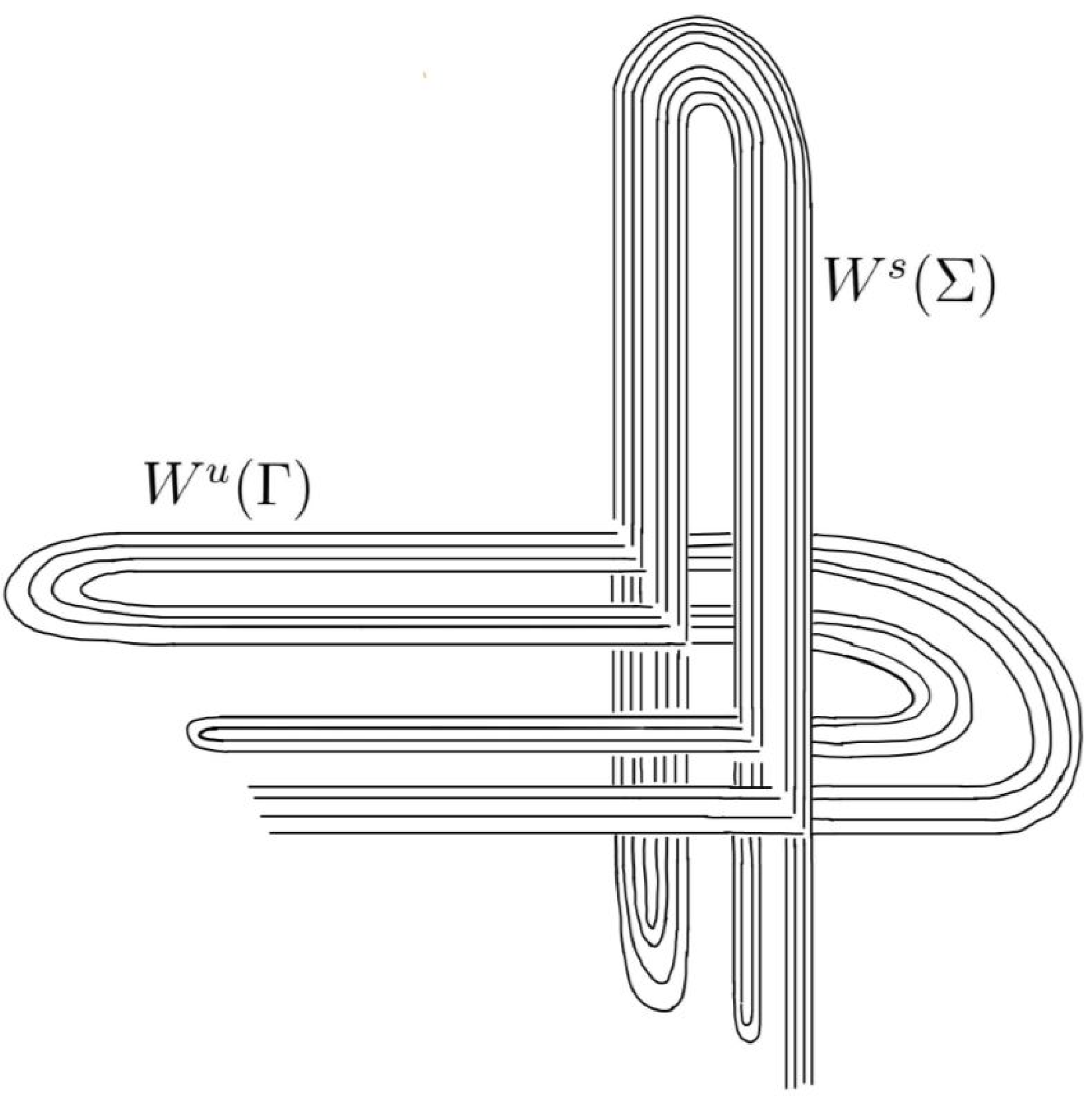}
\caption{The sets $W^u(\Gamma)$, $W^s(\Sigma)$ and their fractal intersection.
}
\label{fig:heteroclinic-intersection}
\end{figure}
It is worthwhile to compare Theorem~C with the result of Moreira and Yoccoz \cite{MorYoc01}, which 
shows that, for generic pairs of regular Cantor sets on the real line with large Hausdorff dimension, one can almost always find translations for which the two translated Cantor sets robustly have
positive Hausdorff dimension. The  thickness is not used in \cite{MorYoc01}.
Our proof of Theorem~C relies on the result of Hunt et al. \cite{HKY93} which asserts that two interleaved Cantor sets with large thicknesses contain a Cantor set with large thickness.

 We do not claim that Theorem~C gives optimal estimates
 on Hausdorff dimension of heteroclinic intersection 
for diffeomorphisms in dimension three. For example,
 modifying the construction of Asaoka \cite{Asa08}, Barrientos and P\'erez \cite{BarPer19} constructed a $C^1$ diffeomorphism having two hyperbolic sets, one of index 1
and the other of index 2, for which the unstable set of the first  set and the stable set of the second one contain  two-dimensional submanifolds which intersect each other in a smooth curve.
As we reiterate, our aim here is to introduce a model and results which
provide a hands-on, elementary understanding of complicated dynamics in dimension three.

\subsection{Outline of proofs of the theorems}\label{outline}

 Proofs of the main results are briefly outlined as follows.
We focus on the Hausdorff dimension estimate of $\mathscr{H}$, since
 that of $W^u(\Gamma)$ and $W^s(\Sigma)$ are by-products. 
We begin by remarking that geometric structures of these sets
for a general heterochaos horseshoe map 
can be quite rich. 
Our numerical experiment suggests that the heteroclinic intersection can contain 
a fractal set as in FIGURE~\ref{fig:indexset-continuous-map}(b),
with Hausdorff dimension seemingly exceeding $1$. Moreover, this fractal set appears to persist under small perturbations of the map.
A dimension estimate in such a case is beyond our reach.

Therefore, we perform a dimension reduction.
The assumption of dimension reducibility allows us to use the theory of normally hyperbolic invariant manifolds
\cite{Fen72,HPS77}
to construct two $C^2$ surfaces, called 
{\it graph-invariant surfaces}, one of which contains
a neighborhood of $\Gamma$ in $W^u(\Gamma)$ and the other contains a neighborhood of $\Sigma$ in $W^s(\Sigma)$
(see
Section~\ref{quasi}).
 The upper bound $\dim_{\rm H} \mathscr{H}<1$
 is a consequence of this construction.

\begin{figure}
    \centering
     \subfigure[]{\includegraphics[height=0.30\textwidth,width=0.27\textwidth]{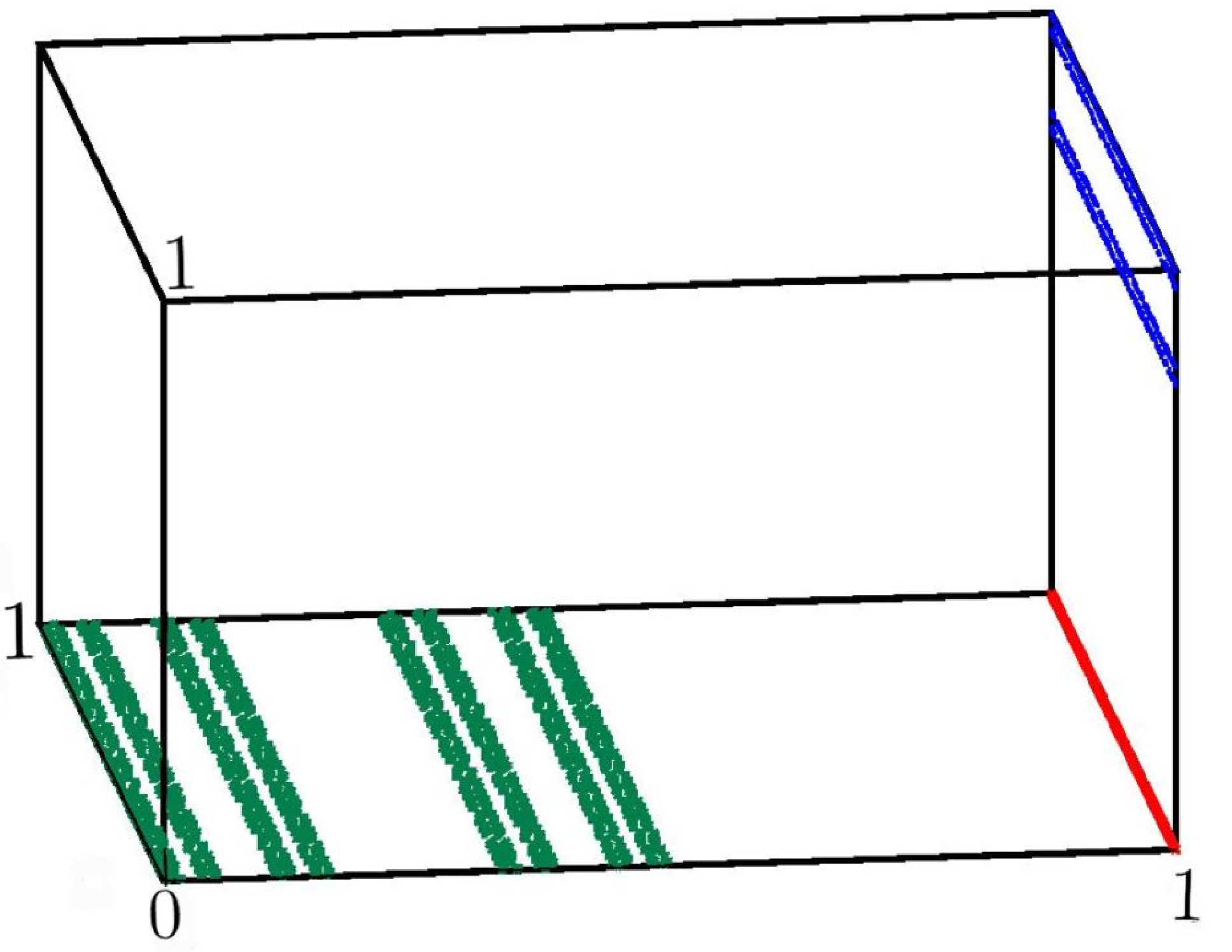}}
\qquad
\qquad
    \subfigure[]{\includegraphics[height=0.30\textwidth,width=0.27\textwidth]{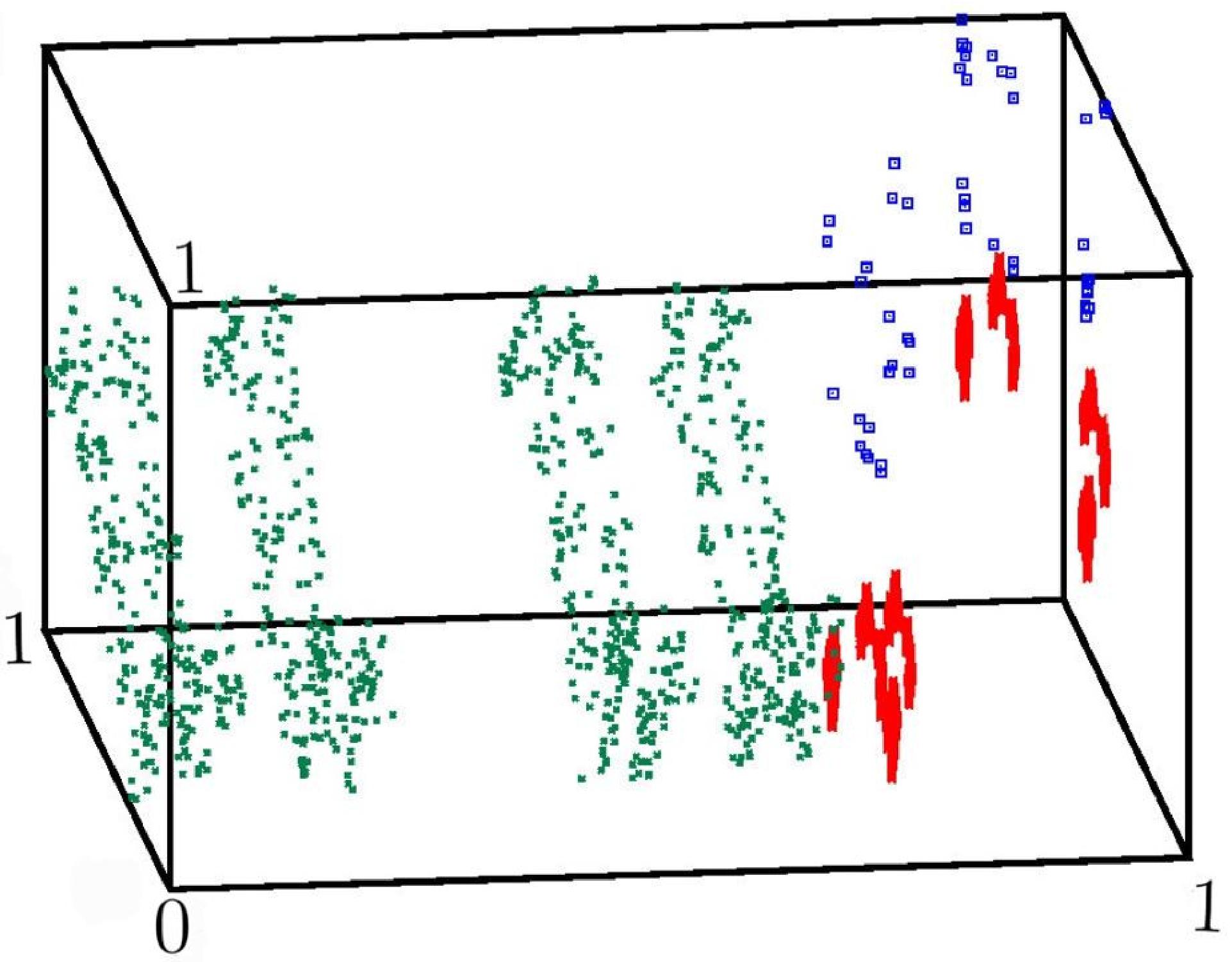}}
   \caption{
   Numerical identification of $\Gamma$ (green), $\Sigma$ (blue) and part of their heteroclinic intersection (red) for heterochaos horseshoe maps: (a) dimension-reducible; (b) not dimension reducible.}
    \label{fig:indexset-continuous-map}
\end{figure}
 
A lower estimate of $\dim_{\rm H} \mathscr{H}$ as well as
 showing $\mathscr{H}\neq\emptyset$
 is more delicate.
We perform a further reduction to the problem of 
how two Cantor sets on the real line intersect each other.
More precisely, we construct two 
Cantor sets contained in a smooth curve,
by projecting the cross sections of $W^u(\Gamma)$ and $W^s(\Sigma)$ 
along their unstable and stable manifolds. Under the assumption of Theorem~A, 
we show that the product
of their thicknesses is strictly bigger than $1$. We then use the result of Newhouse \cite{New70} to conclude
that the two Cantor sets intersect each other, and $\mathscr{H}$ is non-empty.

The number $T_0$ in Theorem~C is a large constant determined by the result of Hunt et al. \cite{HKY93} that
we recall in Section~2.
 We show that, if $T_0$ is sufficiently large,
 then the assumption $\min\left\{a_1,a_2\right\}>T_0$ 
  implies that the two Cantor sets have uniformly large thicknesses.
 By the result of Hunt et al., their intersection contains a Cantor set with uniformly large thickness.
We then appeal to the general lower bound of Newhouse
\cite{New79} on the Hausdorff dimension
of Cantor sets in terms of their thicknesses.

\subsection{Comparison with blender approach}
There is a different line  of research where the intersection analysis of invariant manifolds are
carried out using {\it blenders}.
For a comparison
with our results,
among the many types of blenders, 
the one in \cite{ACW21} has an advantage
in that an affine
horseshoe and its perturbations were also considered.
For other definitions of blenders, see also e.g., \cite{BCDW16,BonDia96} and 
\cite[Section~6.2]{BDV04}.
Key properties of a $d_{cs}$-stable (resp. $d_{cu}$-unstable) blender, in dimension three for example, is that a related one-dimensional local  stable (resp. unstable) manifolds form
a ``topological surface'' which intersects any curve ``transverse'' to it, and that this property is $C^1$-robust. 
These key properties are used to blend $C^1$-robust intersections between one-dimensional invariant manifolds.
For the model presented in the paper, 
 the local stable manifolds of
 $\Sigma$ 
 are contained in a $C^2$ surface
 (see Proposition~\ref{invariance}).
 Hence, $\Sigma$ is not a $d_{cs}$-stable blender: the first key property obviously fails.
 Similarly,
 $\Gamma$ is not a $d_{cu}$-unstable blender since its local unstable manifolds 
 are contained in a $C^2$ surface.

 For heterochaos horseshoe maps 
 satisfying some additional assumptions,
 using blenders
  one can even construct $C^1$-robust heterodimensional cycles associated to hyperbolic sets other than $\Gamma$, $\Sigma$.
  For details, see Appendix.

\subsection{Structure of the paper}
The rest of this paper consists of three sections.
In Section~2 we introduce main tools, and
 in Section~3 perform main constructions.
 In Section~4 we complete the proofs of Theorems~A and C.

\section{Main Tools}
In this section we introduce main tools needed for the proofs of the theorems.
 In Section~\ref{normal} we state a version of
the fundamental theorem of Hirsch et al. \cite{HPS77} on normally hyperbolic invariant manifolds.
In Section~\ref{thick} we introduce the thickness of Cantor sets on the real line, 
and recall the gap lemma of Newhouse \cite{New70} and the theorem of Hunt et al. \cite{HKY93} on when two thick Cantor sets intersect thickly.

\subsection{Persistence of normally attracting invariant manifolds}\label{normal}
Let $M$ be a $C^\infty$ Riemannian manifold and let
$\varphi\colon M\to M$ be a $C^r$ diffeomorphism $(r\geq1)$.
Let $T_p\varphi$ denote the differential of $\varphi$ at $p\in M$.
Let $d$ denote the distance on $M$ given by the Riemannian metric.
The unstable and stable sets of a point $p\in M$ are given by \[\begin{split}
W^u(p)&=\{q\in M\colon d(\varphi^{n}(p),\varphi^{n}(q))\to0\text{ as }n\to-\infty\},\\
W^s(p)&=\{q\in M\colon d(\varphi^{n}(p),\varphi^{n}(q))\to0\text{ as }n\to\infty\},
\end{split}\]
respectively.
If these sets are submanifolds of $M$, they are called {\it unstable} and {\it stable manifolds of $p$}.

A $C^1$ submanifold $V$ of $M$ is called {\it $r$-normally hyperbolic for $\varphi$} if $\varphi(V)=V$, and there exist a continuous $T\varphi$-invariant splitting $T_VM=TV\oplus N^s\oplus N^u$ and constants $c>0$, $\nu\in(0,1)$ such that for any $p\in V$ and all unit vectors
$v\in T_pV$, $n^s\in N_p^s$, $n^u\in N_p^u$ and any $n\geq1$, 
\[\frac{\|T_p\varphi^nn^s\|}{\|T_p\varphi^nv\|^r}\leq c\nu^{n}\quad\text{and}\quad\frac{\|T_p\varphi^{-n}n^u\|}
{\|T_p\varphi^{-n}v\|^r}\leq c\nu^{n},\]
where $\Vert\cdot\Vert$ denotes the norm on the tangent spaces that comes from the Riemannian metric. In the case $N^u=\{0\}$, $V$ is called {\it $r$-normally attracting for $\varphi$}.

Normally hyperbolic invariant manifolds are persistent \cite{Fen72,HPS77}. We recall the precise statement of 
\cite[Theorem~4.1]{HPS77} for the normally attracting case.
The theorem below
asserts that a compact $r$-normally attracting manifold $V$ for a $C^r$ diffeomorphism $\varphi$ is $C^r$, and the stable set $\bigcup_{p\in V}W^s(p)$ of $V$
is invariantly fibered by $C^r$ submanifolds tangent at $V$ to $N^s$.
Moreover, these structures are unique and persistent under $C^r$ perturbations of $\varphi$.

\begin{theorem}[{\rm \cite[Theorem~4.1]{HPS77}}]\label{persistence}
Let $\varphi\colon M\to M$ be a $C^r$ diffeomorphism, $r\geq1$, of a $C^\infty$ Riemannian manifold $M$ leaving a compact $C^1$ submanifold $V$ invariant.
Assume $\varphi$ is $r$-normally attracting at $V$ respecting the splitting
$T_VM=TV\oplus N^s.$
Then $V$ is $C^r$, and the following hold:
\begin{itemize}
\item[(a)] For each $p\in V$ there exists a regular $C^r$ submanifold $W^{ss}(p)$ 
contained in the stable set of $p$ and tangent to $N^s_p$, such that
$\varphi (W^{ss}(p))\subset W^{ss}(\varphi (p))$ for any $p\in V$ and
$\bigcup_{p\in V} W^{ss}(p)$ contains a neighborhood of $V$.

\item[(b)] If $\tilde \varphi$ is another diffeomorphism of $M$ which is sufficiently $C^r$-close to $\varphi$,
then $\tilde \varphi$ is $r$-normally attracting at some unique $\tilde V$, $C^r$-close to $V$.
\end{itemize}
\end{theorem}

\subsection{Intersection of Cantor sets}\label{thick}
We adopt the definition of thickness
by Palis and Takens \cite{PalTak93} that is equivalent 
to the one by Newhouse \cite{New70}.
Let $S$ be a Cantor set in $\mathbb R$.
A {\it gap} of $S$ is a connected component of $\mathbb R\setminus S$.
A {\it bounded gap} is a gap which is bounded.
Let $G$ be any bounded gap and $x$ be a boundary point of $G$.
Let $I$ denote the {\it bridge of $S$ at $x$},
i.e., the maximal interval in $\mathbb R$ that satisfies
 $x\in\partial I$, and contains no point of a gap 
    whose Euclidean length is at least $|G|$.
The thickness of $S$ at $x$ (the local thickness) is defined by
\[\tau(S,x)=\frac{|I|}{|G|}.\] The {\it thickness} $\tau(S)$ of $S$
is the infimum of $\tau(S,x)$ over all boundary points $x$ of bounded gaps.
We say two Cantor sets $S_1$, $S_2$ in $\mathbb R$
are {\it interleaved} if neither set is contained
in the closure of a gap of the other set.

\begin{lemma}[the gap lemma \cite{New70}]\label{gaplem}
Let $S_1, S_2$ be two interleaved Cantor sets in $\mathbb R$ such that $\tau(S_1)\tau(S_2)>1$.
Then $S_1\cap S_2$ is non-empty.
\end{lemma}

The gap lemma asserts that two interleaved Cantor sets on the real line intersect each other if the product of their thicknesses is greater than one.
It does not imply any lower bound of the Hausdorff dimension of the intersection of the two Cantor sets. Indeed,
Williams \cite{Wil91} observed that two interleaved Cantor sets can have thicknesses well above $1$ and still only intersect at a single point. 
The next theorem in \cite{HKY93} asserts that
 the intersection of two interleaved Cantor sets with large thicknesses contains a Cantor set with  large thickness.

\begin{theorem}[{\rm in \cite[p.881, Remark]{HKY93}}]\label{cant-inter}
For any $\varepsilon\in(0,1)$ there exists $T>0$ such that if two Cantor sets $S_1$, $S_2$ in $\mathbb R$ 
with $\tau(S_1)> T$, $\tau(S_2)> T$ are interleaved,
then $S_1\cap S_2$ contains a Cantor set whose thickness is at least
$(1-\varepsilon)\sqrt{\min\{\tau(S_1),\tau(S_2)\}}$.
\end{theorem}

We will deal with Cantor
sets which are subsets of $C^2$ curves in $\mathbb R^3$.
Their thicknesses will be defined with respect to the induced metrics on these curves.

\section{Main constructions}
In this section we perform main constructions needed for 
the proofs of our main results. 
In Section~\ref{constants-s}
we introduce the dimension reducibility which will be assumed throughout the rest of the paper. In Section~\ref{hyp} we provide standard hyperbolicity estimates.  In Section~\ref{quasi} we construct two graph-invariant surfaces, and in Section~\ref{Cantorset} construct two Cantor sets. In Section~\ref{s-c} we describe the structure of these Cantor sets in terms of symbolic dynamics.

\subsection{Initial setup}\label{constants-s}
Let $f_0\in{\rm Diff}^1(\mathbb R^3)$ be a  heterochaos horseshoe map. Define
\begin{equation}\label{constant1}\lambda_{0,i}
=\frac{|A_i^*|}{|A_i|}\quad\text{and}\quad
\nu_{0,i}=\frac{|B_i^*|}{|B_i|}\quad \text{ for }i\in\{u,c,s\}.
\end{equation}
We say $f_0$  is {\it dimension-reducible} if 
$f_0 \in{\rm Diff}^2(\mathbb R^3)$,
  all the diagonal elements of its Jacobian matrices on $A$, $B$, $C$, $D$ are positive,
  and
\begin{equation}\label{d-reduce}\lambda_{0,c}^2>\lambda_{0,s}\quad\text{and}\quad 
\nu_{0,u}>\nu_{0,c}^{2}.\end{equation}

\noindent{\bf  (Standing hypothesis for the rest of the paper)}: 
 $f_0\in{\rm Diff}^2(\mathbb R^3)$ is a 
  heterochaos horseshoe map that is dimension-reducible.\medskip
  
In particular, we have
\begin{equation}\label{condition1}\lambda_{0,u}>2>1/2>\lambda_{0,c}>\lambda_{0,s}
\quad\text{and}\quad\nu_{0,u}>\nu_{0,c}>2>1/2>\nu_{0,s}.\end{equation}

\subsection{Hyperbolic behaviors}\label{hyp}
We use the set 
$\{\partial_{x_u},\partial_{ x_c},\partial_{x_s}\}$ of the first-order partial derivatives as a basis for the tangent space
at a point of $\mathbb R^3$. 
For $\theta>0$ and $p\in\mathbb R^3$,
we introduce the following
subsets of $T_p\mathbb R^3$:
\[\begin{split}
\mathcal C^{u}_p(\theta)&=\left\{(\xi,\eta,\zeta)
\colon \theta|\xi|\geq\sqrt{|\eta|^2+|\zeta|^2}\right\},\  \mathcal C^{cu}_p(\theta)=\left\{(\xi,\eta,\zeta)\colon \theta\sqrt{|\xi|^2+|\eta|^2}\geq|\zeta|\right\},\\ 
\mathcal C^{s}_p(\theta)&=\left\{(\xi,\eta,\zeta)\colon \sqrt{|\xi|^2+|\eta|^2}\leq\theta|\zeta|\right\},\ \mathcal C^{cs}_p(\theta)=\left\{(\xi,\eta,\zeta)\colon|\xi|\leq \theta\sqrt{|\eta|^2+|\zeta|^2}\right\},\\
\mathcal C^{c}_p(\theta)&=
\mathcal C^{cu}_p(\theta)\cap\mathcal C^{cs}_p(\theta).
\end{split}\]
Note that $\mathcal C^u_p(\theta)\subset\mathcal C^{cu}_p(\theta)$
and
$\mathcal C^s_p(\theta)\subset\mathcal C^{cs}_p(\theta)$.
If $\theta\in(0,1/10)$ then
 $\mathcal C^{u}_p(\theta)\cap
\mathcal C^{cs}_p(\theta)=\{0\}=\mathcal C^{cu}_p(\theta)\cap
\mathcal C^{s}_p(\theta)$.

By a {\it $C^r$ curve} (resp. {\it surface})
we mean a regular one-(resp. two-)dimensional $C^r$ submanifold of $\mathbb R^3$.
For $i\in\{u,c,s\}$, 
we say a $C^1$ curve $\gamma$ is {\it tangent to} $\mathcal C^i(\theta)$ 
if $T_p\gamma\subset\mathcal C^i_p(\theta)$ holds for any $p\in\gamma$. For $i\in\{cu,cs\}$, 
we say
a $C^1$ surface $V$ is {\it tangent to} $\mathcal C^{i}(\theta)$ 
if $T_pV\subset\mathcal C_p^{i}(\theta)$ holds
for any $p\in V$.

For $p\in\mathbb R^3$ and a subset $\mathcal C$ of $T_p\mathbb R^3$,
define 
\[\|T_pf\|_{\mathcal C}=\sup_{v\in \mathcal C\setminus\{0\}}\frac{\|T_pfv\|}{\|v\|}
\quad\text{
and }\quad|T_pf|_{\mathcal C}=\inf_{v\in \mathcal C\setminus\{0\}}\frac{\|T_pfv\|}{\|v\|},\]
where $\|\cdot\|$ denotes the Euclidean norm on the tangent space of $\mathbb R^3$.

The next two propositions are consequences of the affinity of $f_0$ on the four blocks with
the expansion and contraction rates
in \eqref{constant1} satisfying \eqref{condition1}.

\begin{prop}\label{tangent}
For any
 $\theta\in(0,1/10)$, there exist two
 disjoint open sets $U_1\supset R_1$, $U_2\supset R_2$ 
 such that for any 
 $f\in{\rm Diff}^1(\mathbb R^3)$ which is sufficiently $C^1$-close to $f_0$ and for $i=1,2$ the following hold:
 \begin{itemize}
\item[(a)] 
If $\gamma\subset U_i$ is a $C^1$ curve
which is tangent to $\mathcal C^u(\theta)$ (resp. $\mathcal C^{s}(\theta)$),
then $f(\gamma)$ (resp. $f^{-1}(\gamma)$) is tangent to $\mathcal C^u(\theta)$
(resp. $\mathcal C^{s}(\theta)$).

\item[(b)] If $V\subset U_i$ is a $C^1$ surface which is tangent to $\mathcal C^{cu}(\theta)$ (resp. $\mathcal C^{cs}(\theta)$),
then $f(V)$ (resp. $f^{-1}(V)$) is tangent to $\mathcal C^{cu}(\theta)$
(resp. $\mathcal C^{cs}(\theta)$).
\end{itemize}
 \end{prop}

\begin{prop}\label{cone-cor}
Let $\lambda_u$, $\lambda_c$, $\lambda_c'$, $\lambda_s$, $\nu_u$, $\nu_c$, $\nu_c'$, $\nu_s\in\mathbb R$ satisfy
\begin{equation}\label{condition10}\lambda_{0,u}>\lambda_u>2>1/2>\lambda_c>\lambda_{0,c}>\lambda_c'>\lambda_s>\lambda_{0,s}
\quad\text{and}\end{equation}
\begin{equation}\label{condition20}\nu_{0,u}>\nu_u>\nu_c>\nu_{0,c}>\nu_c'>2>1/2>\nu_s>\nu_{0,s}.\end{equation}
There exist $\theta\in(0,1/10)$ and two
 disjoint open sets $U_1\supset R_1$, $U_2\supset R_2$ 
 such that for any 
 $f\in{\rm Diff}^1(\mathbb R^3)$ which is sufficiently $C^1$-close to $f_0$ the following hold:
 \begin{itemize}
\item[(a)] If $p\in U_1$ then    \begin{equation}\label{cone-eq1}|T_pf|_{\mathcal C_{p}^{u}(\theta)}\geq\lambda_{u}\quad\text{and}\quad
         |T_{f(p)}f^{-1}|_{\mathcal C_{f(p)}^{s}(\theta)}\geq\lambda_{s}^{-1},\end{equation}   
       \begin{equation}\label{cone-eq2}
      |T_{f(p)}f^{-1}|_{\mathcal C^{c}_{f(p)}
      (\theta)}{\geq\lambda_{c}}^{-1}\quad\text{and}
      \quad\|T_{f(p)}f^{-1}\|_{\mathcal C^{c}_{f(p)}(\theta) }
       \leq{\lambda_{c}'}^{-1}.\end{equation}
  \item[(b)] If $p\in U_2$ then    \begin{equation}\label{cone-eq3}|T_pf^{-1}|_{\mathcal C_{p}^{s}(\theta)}\geq\nu_{s}^{-1}\quad\text{and}\quad
  |T_{f^{-1}(p)}f|_{\mathcal C_{f^{-1}(p)}^{u}(\theta)}\geq\nu_{u},\end{equation}  
       \begin{equation}\label{cone-eq4}|T_{f^{-1}(p)}f|_{\mathcal C^{c}_{f^{-1}(p)}(\theta) }\geq\nu_c'\quad\text{and}\quad\|T_{f^{-1}(p)}f\|_{\mathcal C^{c}_{f^{-1}(p)}(\theta) }
       \leq\nu_{c}.\end{equation}
       \end{itemize}
\end{prop}

\subsection{Construction of graph-invariant  surfaces}\label{quasi}
We prove a dimension reduction result announced in Section~\ref{outline}. A key concept is that of graph-invariant surfaces.
Since we deal with two horseshoe maps with different unstable dimensions, 
there are two different types of graph-invariant surfaces. 
Let
 \[\pi_1\colon (x_u,x_c,x_s)\mapsto(x_u,x_c)\quad\text{and}\quad\pi_2\colon (x_u,x_c,x_s)\mapsto(x_c,x_s)\] be the natural projections, and
 let $f\in{\rm Diff}^2(\mathbb R^3)$. A $C^2$ surface $V^{cu}$ is 
called {\it graph-invariant 
in $R_1$ for $f$} if 
there exists
a $C^2$ function $\phi_1\colon
\pi_1(R_1)\to \mathbb R$,
called the {\it defining function} with the following properties (see FIGURE~\ref{fig:graphinvariantsurface}):
\begin{itemize}
\item[(i)] $V^{cu}=\{(x_u,x_c,\phi_1(x_u,x_c))\colon(x_u,x_c)\in \pi_1(R_1)\}$.
\item[(ii)] 
$V^{cu}\cap f^{-1}(R_1)\subset f^{-1}(V^{cu})$ and
$V^{cu}\cap f(R_1)\subset f(V^{cu})$.

\item[(iii)] 
$\sum_{\alpha:1\leq|\alpha|\leq 2}\sup_{\pi_1(R_1)}|\partial^\alpha\phi_1|<1/10$.
\end{itemize}
where $\alpha=(\alpha_u,\alpha_c,\alpha_s)$ denotes the multi-index:
$\partial^\alpha=\partial^{\alpha_u}_{x_u}\partial^{\alpha_c}_{x_c}\partial^{\alpha_s}_{x_s}$.
Similarly,
a $C^2$ surface $V^{cs}$ is 
called {\it graph-invariant in $R_2$ for $f^{-1}$} if 
there exists
a $C^2$ function $\phi_2\colon
\pi_2(R_2)\to \mathbb R$ with the following properties:
\begin{itemize}
\item[(i)] $V^{cs}=\{(\phi_2(x_c,x_s),x_c,x_s)\colon (x_c,x_s)\in \pi_2(R_2)\}$.
\item[(ii)] 
$V^{cs}\cap f^{-1}(R_2)\subset f^{-1}(V^{cs})$ and 
$V^{cs}\cap f(R_2)\subset f(V^{cs})$.

\item[(iii)] $\sum_{\alpha:1\leq|\alpha|\leq 2}\sup_{\pi_2(R_2)}|\partial^\alpha\phi_2|<1/10$.
\end{itemize}

 \begin{figure}
\centering
\includegraphics[height=0.22\textwidth,width=0.90\textwidth]{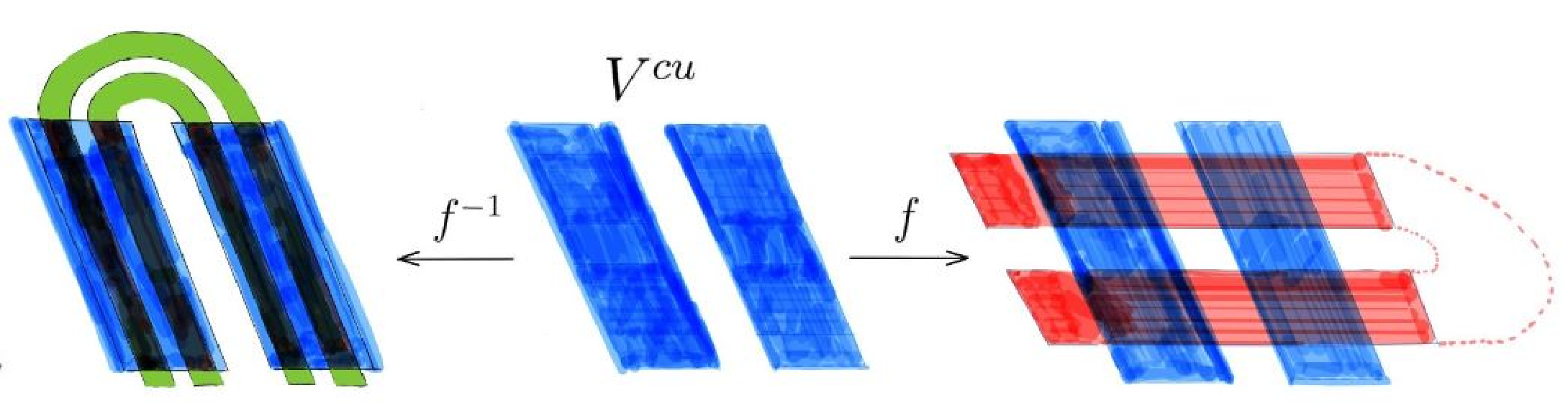}
\caption{The graph-invariant surface $V^{cu}$ in $R_1$ for $f$ (blue), its forward image $f(V^{cu})$ (red) and backward image $f^{-1}(V^{cu})$(green).}
\label{fig:graphinvariantsurface}
\end{figure}

\begin{prop}
\label{invariance}
If $\theta>0$ is sufficiently small and $f\in{\rm Diff}^2(\mathbb R^3)$ is 
sufficiently $C^2$-close to $f_0$,
there exists a unique
 graph-invariant surface $V^{cu}$ in $R_1$ for $f$. Moreover we have
 $\bigcap_{n=0}^{\infty}f^n(R_1)\subset V^{cu}$,
 and the defining function $\phi_1$ of $V^{cu}$ satisfies $\sum_{\alpha:1\leq|\alpha|\leq 2}\sup_{\pi_1(R_1)}|\partial^\alpha\phi_1|<\theta$.
Similarly, there exists a unique  graph-invariant surface $V^{cs}$
in $R_2$ for $f^{-1}$. Moreover we have
 $\bigcap_{n=0}^{\infty}f^{-n}(R_2)\subset V^{cs}$,
 and the defining function
 $\phi_2$ of $V^{cs}$ satisfies 
 $\sum_{\alpha:1\leq|\alpha|\leq 2}\sup_{\pi_2(R_2)}|\partial^\alpha\phi_2|<\theta$.
\end{prop}

\begin{proof}
Let $\theta>0$ be sufficiently small,
let 
$\lambda_u$, $\lambda_c$, $\lambda_c'$, $\lambda_s$, $\nu_u$, $\nu_c$, $\nu_c'$, $\nu_s\in\mathbb R$
satisfy \eqref{condition10}, \eqref{condition20}, let
 $U_1\supset R_1$, $U_2\supset R_2$
be disjoint open sets and let $f$ be sufficiently $C^2$-close to $f_0$
for which the conclusions of Propositions~\ref{tangent} and \ref{cone-cor} hold.
We only give a proof of the statements on $V^{cu}$ since that on $V^{cs}$ is analogous.

Let $x_{0,s}$ denote the $x_s$-coordinate of the fixed saddle of $f_0$ in $A$.
 Since $f_0$ is dimension-reducible, the first inequality in \eqref{d-reduce} implies that
there exist a
$C^2$ surface $M$ diffeomorphic to the two-dimensional sphere 
 and $\tilde f_0\in{\rm Diff}^2(\mathbb R^3)$
 such that
 $\pi_1(R_1)\times \{x_{0,s}\}\subset M$,
  $\tilde f_0|_{U_1}=f_0|_{U_1}$,
and $M$ is $2$-normally attracting for
$\tilde f_0$.
 By Theorem~\ref{persistence} applied to $(\tilde f_0,M)$,
 there exist a $C^2$ neighborhood $\mathcal Y$ of $\tilde f_0$,
 an open set $\Delta\subset\mathbb R^2$ containing 
   $\pi_1(R_1)$, a constant $\delta_0>0$,
 and for each $g\in\mathcal Y$ there exist a
 $2$-normally attracting $C^2$ submanifold $V(g)$ 
  for $g$, 
  and a $C^2$ function $\phi(g)\colon \Delta\to \mathbb R$ 
  with the following properties:
  \begin{itemize}
      \item[(i)]
  ${\rm graph}(\phi(g))=\{(x_u,x_c,\phi(g)(x_u,x_c))\in\mathbb R^3\colon (x_u,x_c)\in \Delta\}\subset V(g)$.
 
 \item[(ii)]
 $\sum_{\alpha:1\leq|\alpha|\leq 2}\sup_{\Delta}|\partial^\alpha\phi(g)|<\theta$.
 
 \item[(iii)]
For each $p\in {\rm graph}(\phi(g))$
there exists a 
$C^1$ curve $W^{ss}(p)$ through $p$ which is 
tangent to $\mathcal C^s(\theta)$ such that 
 $g(p)\in {\rm graph}(\phi(g))$
implies $g (W^{ss}(p))\subset W^{ss}(g(p))$, and
\begin{equation}\label{include} \pi_1(R_1)\times(x_{0,s}-\delta_0,x_{0,s}+\delta_0)\subset W^{ss}({\rm graph}(\phi(g)))\subset U_1,\end{equation}
where
$W^{ss}({\rm graph}(\phi(g)))=\bigcup\{W^{ss}(p)\colon p\in {\rm graph}(\phi(g))\}$.
\end{itemize}

For each $f$ which is sufficiently $C^2$-close to $f_0$, 
we choose $g_f\in\mathcal Y$ such that $f|_{U_1}=g_f|_{U_1}$,
and set \[V^{cu}={\rm graph}(\phi(g_f))\cap R_1.\]
Then
$V^{cu}$ is a graph-invariant surface in $R_1$ for $f$
whose defining function is the restriction of $\phi(g_f)$ to $\pi_1(R_1)$.

\begin{lemma}\label{lem-1}
We have $\bigcap_{n=0}^{\infty}f^n(R_1)\subset W^{ss}({\rm graph}(\phi(g_f)))$.
\end{lemma}
\begin{proof}
Since the block $D$ is a translate of $A$ in the $x_u$-direction,
$\bigcap_{n=0}^{\infty}f_0^n(R_1)$ is contained in $\pi_1(R_1)\times\{x_{0,s}\}$.
From this and \eqref{include} the desired inclusion follows.
\end{proof}

\begin{lemma}\label{lem-2}
Any $C^1$ curve which is tangent to $\mathcal C^s(\theta)$ and intersects
$\bigcap_{n=0}^2f^n(R_1)$ does not intersect $V^{cu}\setminus f(R_1)$.\end{lemma}
\begin{proof}
Any $C^1$ curve which is tangent to $\mathcal C^s(2\theta)$ and intersects
$\bigcap_{n=0}^2f_0^n(R_1)$ does not intersect $(R_1\cap\{x_s=x_{0,s}\})\setminus f_0(R_1)$. This
 implies the assertion of the lemma.
\end{proof}

We now prove $\bigcap_{n=0}^{\infty}f^n(R_1)\subset V^{cu}$ 
by way of contradiction.
Let $p\in\bigcap_{n=0}^{\infty}f^n(R_1)$ and suppose
 $p\notin V^{cu}$. From Lemma~\ref{lem-1},
there exists $q\in{\rm graph}(\phi(g_f))$
such that $p\in W^{ss}(q)$.
Since $p\notin V^{cu}$ we have $p\neq q$.
Let $\sigma$ denote the $C^1$ curve in $W^{ss}(q)$ joining $p$ and $q$.
 The images of $\sigma$
 under the iteration of $f^{-1}$ are tangent to $\mathcal C^s(\theta)$ by Proposition~\ref{tangent}, and expanded by factor $\lambda_s^{-1}$
 by Proposition~\ref{cone-cor}. 
  Let $n_0\geq1$ denote the minimal integer such that
 $f^{-n_0}(\sigma)\notin W^{ss}({\rm graph}(\phi(g_f)))$. 
 Lemma~\ref{lem-1} implies $f^{-n}(p)\in W^{ss}({\rm graph}(\phi(g_f)))$ for all $n\geq0$, and so
 we must have 
 $\{f^{-n_0+1}(q)\}\setminus\{f^{-n_0}(q)\}\subset {\rm graph}(\phi(g_f))$. 
It follows that $f^{-n_0+1}(p)\in \bigcap_{n=0}^2f^n(R_1)$
and $f^{-n_0+1}(q)\in V^{cu}\setminus f(R_1)$. In particular,
$f^{-n_0+1}(\sigma)$ intersects $\bigcap_{n=0}^2f^n(R_1)$
 and $V^{cu}\setminus f(R_1)$. 
  By Lemma~\ref{lem-2},
 $f^{-n_0+1}(\sigma)$ is not tangent to $\mathcal C^s(\theta)$,
 a contradiction to Proposition~\ref{tangent}.

To show the uniqueness, 
 let $V$, $V'$ be 
 graph-invariant surfaces in $R_1$ for $f$.
 The argument in the previous paragraph shows that $V\cap V'\supset\bigcap_{n=0}^{\infty}f^n(R_1)$.
 The uniqueness in Theorem~\ref{persistence} implies
 $V=V'$.
 \end{proof}

\subsection{Construction of Cantor sets}\label{Cantorset}
 Throughout the rest of this section, let $\theta>0$ be sufficiently small and let $f$
 be sufficiently $C^2$-close to $f_0$ for which the conclusions of Propositions~\ref{tangent}, \ref{cone-cor}, \ref{invariance} hold.
   Let $P$, $Q$ denote the fixed saddles of $f$ in $A$, $C^*$ respectively. 
Let $W^s_{\rm loc}(P)$
denote the connected component of 
$W^s(P)\cap A$ containing $P$, and
let $W^u_{\rm loc}(Q)$ denote the connected component of $W^u(Q)\cap C^*$ 
containing $Q$.
Put
 \[\gamma_c=W^s_{\rm loc}(P)\cap V^{cu}
\quad\text{and}\quad \sigma_c=W^u_{\rm loc}(Q)\cap V^{cs}.\]
See FIGURE~\ref{fig:heteroclinic-intersection2}.
The curve $\gamma_c$ is $C^2$ since it is a transverse intersection of two $C^2$ surfaces $W^s_{\rm loc}(P)$ and $V^{cu}$.
For the same reason, $\sigma_c$
is a $C^2$ curve.
We fix two $C^2$ curves 
$\tilde\gamma_c\supset\gamma_c$,
$\tilde\sigma_c\supset\sigma_c$
such that 
$\tilde\gamma_c\setminus\gamma_c$
and
$\tilde\sigma_c\setminus\sigma_c$
have two connected components both of infinite length, and view $\gamma_c\cap\Gamma$ and $\sigma_c\cap\Sigma$ as Cantor sets with respect to the induced metrics
on $\tilde\gamma_c$ and $\tilde\sigma_c$.

\begin{figure}
\centering
\includegraphics[height=0.45\textwidth,width=0.75\textwidth]
{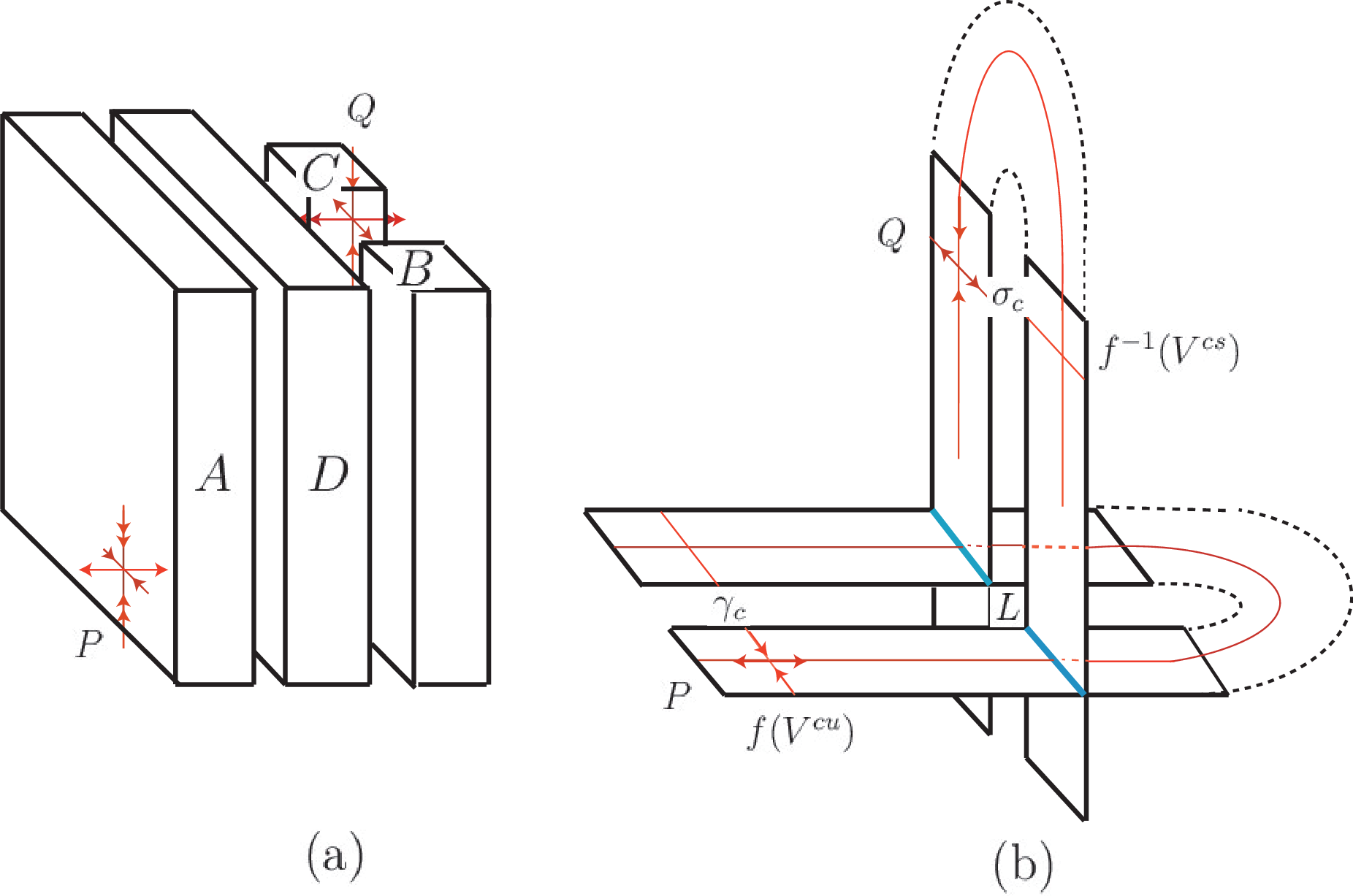}
\caption{(a) the fixed saddles $P$ and $Q$; (b) the surfaces $f(V^{cu})$ and $f^{-1}(V^{cs})$.}
\label{fig:heteroclinic-intersection2}
\end{figure}

For $p\in f(V^{cu})$,
let $\mathcal F^u(p)$ denote the connected component of  $W^u(p)\cap f(V^{cu})$ which contains $p$. 
For $q\in f^{-1}(V^{cs})$,
let $\mathcal F^s(q)$ denote the connected component of $W^s(q)\cap f^{-1}(V^{cs})$ which contains $q$.
Set
\[\mathcal F^u(\Gamma)=\bigcup_{ p\in\gamma_c\cap\Gamma }\mathcal F^u(p)\quad \text{and}\quad \mathcal F^s(\Sigma)=\bigcup_{q\in\sigma_c\cap\Sigma}\mathcal F^s(q).\]
By Proposition~\ref{invariance}, for any $i\geq0$ we have $f^{-i}(\mathcal F^u(\Gamma))\subset \bigcap_{n=0}^{\infty}f^n(R_1)\subset V^{cu}$ and $f^{i}(\mathcal F^s(\Sigma))\subset \bigcap_{n=0}^{\infty}f^{-n}(R_2)\subset V^{cs}$, namely 
\begin{equation}\label{include-f}\mathcal F^u(\Gamma)\subset\bigcap_{n=0}^\infty f^n(V^{cu})\quad\text{and}\quad\mathcal F^s(\Sigma)\subset\bigcap_{n=0}^\infty f^{-n}(V^{cs}).
\end{equation}
The set 
\[L=f(V^{cu})\cap f^{-1}(V^{cs})\]
contains $\mathcal F^u(\Gamma)\cap \mathcal F^s(\Sigma)$, and
 consists of two connected 
components.
Since any connected component of $L$ is a transverse intersection between two $C^{2}$ surfaces,
it is a $C^{2}$ curve. By Proposition~\ref{tangent}, 
the connected components of $L$ are tangent to $\mathcal C^{c}(\theta)$.
We fix a $C^2$ curve $\tilde L$ 
which is tangent to $\mathcal C^{c}(\theta)$ and contains $L$ so that $\tilde L\setminus L$
has two connected components of infinite length.

 We now define two projections $\Pi^u\colon \mathcal F^u(\Gamma)\to L$ and
$\Pi^s\colon \mathcal F^s(\Sigma)\to L$
by 
\begin{equation}\label{proj-us}\Pi^u(p)\in L\cap\mathcal{F}^u(p)\quad\text{and}\quad
\Pi^s(q)\in L\cap\mathcal{F}^s(q),\end{equation}
and set
\begin{equation}\label{omega12}\Omega_1=\Pi^u(\gamma_c\cap \Gamma)
\quad \text{and}\quad \Omega_2=\Pi^s(\sigma_c\cap \Sigma)
.\end{equation}
We view $\Omega_1$ and $\Omega_2$ as Cantor sets with respect to the induced metric on $\tilde L$, with thicknesses $\tau(\Omega_1)$ and $\tau(\Omega_2)$. Clearly,
 $\Omega_1$ and $\Omega_2$ are interleaved and satisfy $\Omega_1\subset W^u(\Gamma)$, $\Omega_2\subset W^s(\Sigma)$, $\Omega_1\cap \Omega_2\subset \mathscr{H}$.

\subsection{Symbolic coding and description of Cantor sets}\label{s-c}
In order to further describe the structure of the set
$\gamma_c\cap\Gamma$,
it is convenient to use a coding on two symbols $f(A),f(D)$.
For each $n\geq1$ and
a word $\omega^-=\omega_{-n}\cdots \omega_{-1}$ in $\{f(A),f(D)\}^n$, define 
an {\it $n$-cylinder in $\gamma_c$} by
\[[\omega^-]=\{p\in\gamma_c\colon f^{-i}(p)\in \omega_{-i-1}\quad\text{for } 0\leq  i\leq n-1\}.\]
Using 
 the condition $V^{cu}\cap f(R_1)\subset f(V^{cu})$ inductively, one can show that 
 $f^{-n}([\omega^-])\subset V^{cu}$ holds for any $n\geq1$ and any $\omega^-\in\{f(A),f(D)\}^n$. From this and Proposition~\ref{tangent}, for any $n\geq1$ and any $\omega^-\in\{f(A),f(D)\}^n$, 
$f^{-n}([\omega^-])$
is tangent to $\mathcal C^c(\theta)$  and stretches across one of the connected components of $V^{cu}$ (see FIGURE~\ref{fig:cylinder-im}).
 Hence,
each $n$-cylinder $[\omega^-]$ contains exactly two $(n+1)$-cylinders
$[ f(A)\omega^-]$ and $[ f(D)\omega^-]$,
and the set $[\omega^-]\setminus([ f(A)\omega^-]\cup [ f(D)\omega^-])$
has exactly three connected components. 

\begin{figure}
\centering
\includegraphics[height=0.5\textwidth,width=0.7\textwidth]
{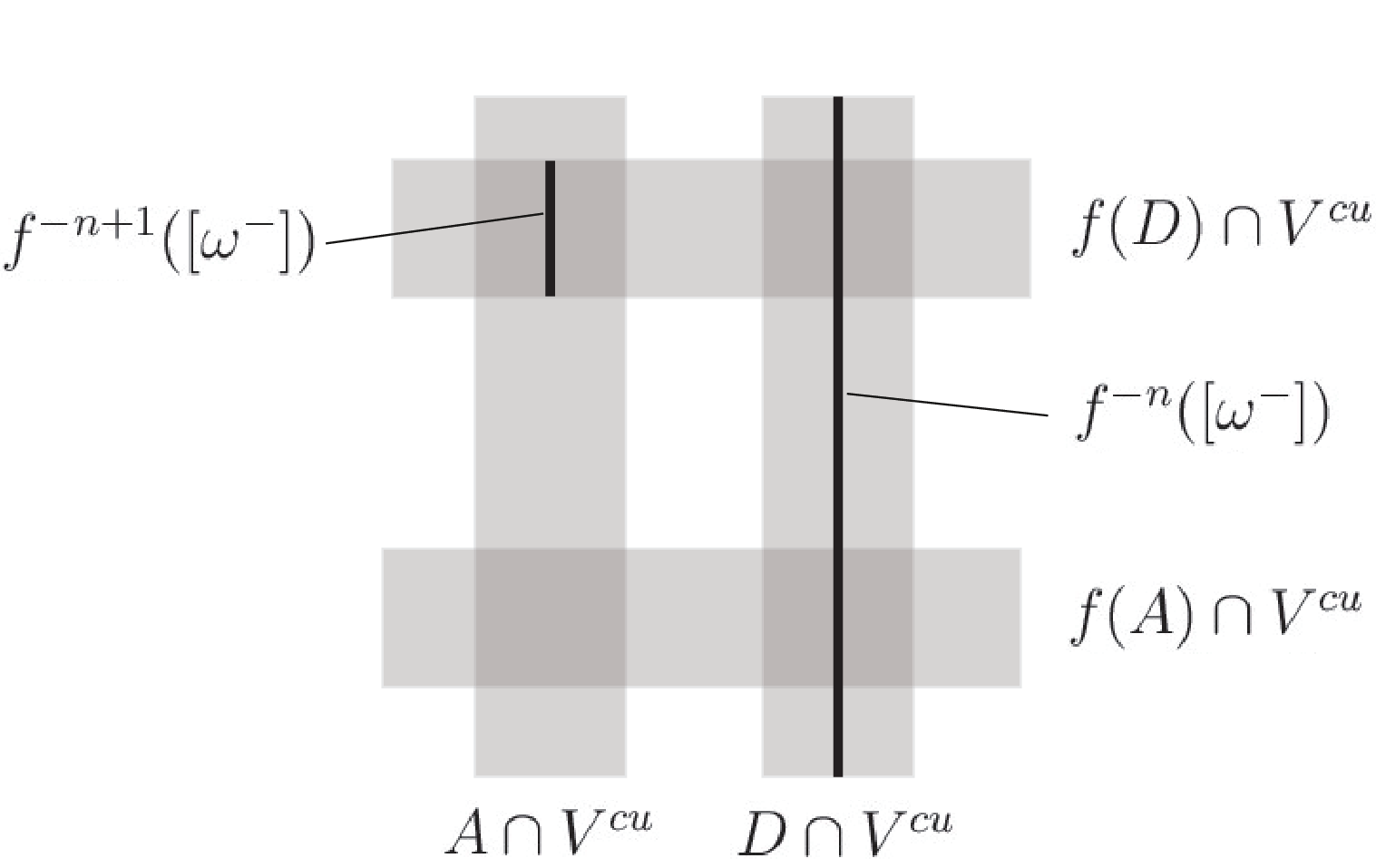}
\caption{The image of an $n$-cylinder $[\omega^-]$ in $\gamma_c$,
 $\omega^-=\omega_{-n}\cdots\omega_{-1}$, $\omega_{-n}=f(A)$.}
 \label{fig:cylinder-im}
\end{figure}

Similarly, each point in $\sigma_c\cap\Sigma$ is uniquely coded by a sequence of two symbols $f^{-1}(B^*)$ and $f^{-1}(C^*)$ which records the history of the forward orbit of the point.
For each $n\geq1$ and a word $\omega^+=\omega_1\cdots \omega_{n}$ in
 $\{f^{-1}(B^*),f^{-1}(C^*)\}^n$, 
define 
an {\it $n$-cylinder in $\sigma_c$} by
\[[\omega^+]=\{q\in\sigma_c\colon f^{i}(q)\in \omega_{i+1}\quad\text{for }0\leq i\leq n-1\}.\]

It may be useful to view $\gamma_c\cap\Gamma$  
as constructed by an inductive procedure
like the middle third Cantor set: starting from $\gamma_c^0=\gamma_c\cap f(V^{cu})$, at step $n\geq1$ we construct an $n$-th approximation $\gamma_c^{n}= \gamma_c\cap\bigcap_{i=0}^{n} f^{i}(f(V^{cu}))$
by deleting from $f^{-n}(\gamma_c^{n-1})$
all those points not contained in $f(V^{cu})$, and then applying $f^n$ to the remaining set. 
For each $n\geq1$, 
 $\gamma_c^{n-1}$ consists of $2^{n}$ connected components which are $n$-cylinders
 (see FIGURE~\ref{fig:cylinder-im-old}),
 and we have $\gamma_c\cap\Gamma=\bigcap_{n=0}^{\infty} \gamma_c^{n}$.
Unlike the middle third Cantor set, we must delete from each connected component of $\gamma_c^{n-1}$ some
neighborhoods of its boundary points.
Hence, any gap of $\gamma_c\cap\Gamma$
is not created in a finite step of the induction, and
a careful analysis is required
to estimate the thickness of $\gamma_c\cap\Gamma$.

For each bounded gap $G$ of $\gamma_c\cap\Gamma$, we associate a non-negative integer $t_G$ as follows.
If $G$ is the unique bounded gap of $\gamma_c\cap\Gamma$ that intersect both $f(A)$ and $f(D)$, we set $t_G=0$.
Otherwise, either
$G\subset f(A)$ or $G\subset f(D)$ holds.
We set \[t_{G}=
\max\{n\geq1\colon\text{there exists an $n$-cylinder in $\gamma_c$ containing $G$}\}.\]
If $t_G\geq1$, then
 $f^{-i}(G)\subset f(V^{cu})$ for $0\leq i\leq t_G-1$ and $f^{-t_G}(G)\not\subset f(V^{cu})$.

If $t_G\geq1$,
let $\omega^-(G)=\omega_{-t_G}\cdots\omega_{-1}$ denote the word in $\{f(A),f(D)\}^{t_G}$ such that 
$G\subset[\omega^-(G)]$.
If $t_G=0$, we define $\omega^-(G)$ to be the empty word in $\bigcup_{n=1}^\infty\{f(A),f(D)\}^{n}$,
and set $[\omega^-(G)]=\gamma_c$.
Any bounded gap $G'$ of $\gamma_c\cap\Gamma$ in $[\omega^-(G)]$ other than $G$ satisfies $t_{G'}>t_G$.
Note that $f^{-t_G}([\omega^-(G)])$  stretches across 
$A\cap V^{cu}$ or $D\cap V^{cu}$. 
Let $\underline{G}$ denote 
 the maximal curve in  $G$ such that  $f^{-t_G}(\underline{G})\cap f(V^{cu})=\emptyset$.
 The set $G\setminus\underline{G}$
has exactly two connected components. If $t_G\geq1$, 
 $\underline{G}$ is
the middle connected component of
 $[\omega^-(G)]\setminus([ f(A)\omega^-(G)]\cup [ f(D)\omega^-(G)])$. 

Similarly, for each bounded gap $H$ of $\sigma_c\cap\Sigma$ we associate a non-negative integer $u_H$ as follows.
If $H$ is the unique bounded gap of $\sigma_c\cap\Sigma$ that intersect both $f^{-1}(B^*)$ and $f^{-1}(C^*)$, we set $u_H=0$.
Otherwise, 
we set \[u_H=
\max\{n\geq1\colon\text{there exists an $n$-cylinder in $\sigma_c$ containing $H$}\}.\]
If $u_H\geq1$,
let $\omega^+(H)=\omega_1\cdots\omega_{u_H}$ denote the word in 
$\{f^{-1}(B^*),f^{-1}(C^*)\}^{u_H}$ such that 
$H\subset[\omega^+(H)]$.
If $u_H=0$, we define $\omega^+(H)$ to be the empty word in $\bigcup_{n=1}^\infty\{f^{-1}(B^*),f^{-1}(C^*)\}^{n}$,
and set $[\omega^+(H)]=\sigma_c$. 
Let $\underline{H}$ denote 
 the maximal curve in  $H$ such that  $f^{u_H}(\underline{H})\cap f^{-1}(V^{cs})=\emptyset$.

\begin{figure}
\centering
\includegraphics[height=0.2\textwidth,width=0.75\textwidth]
{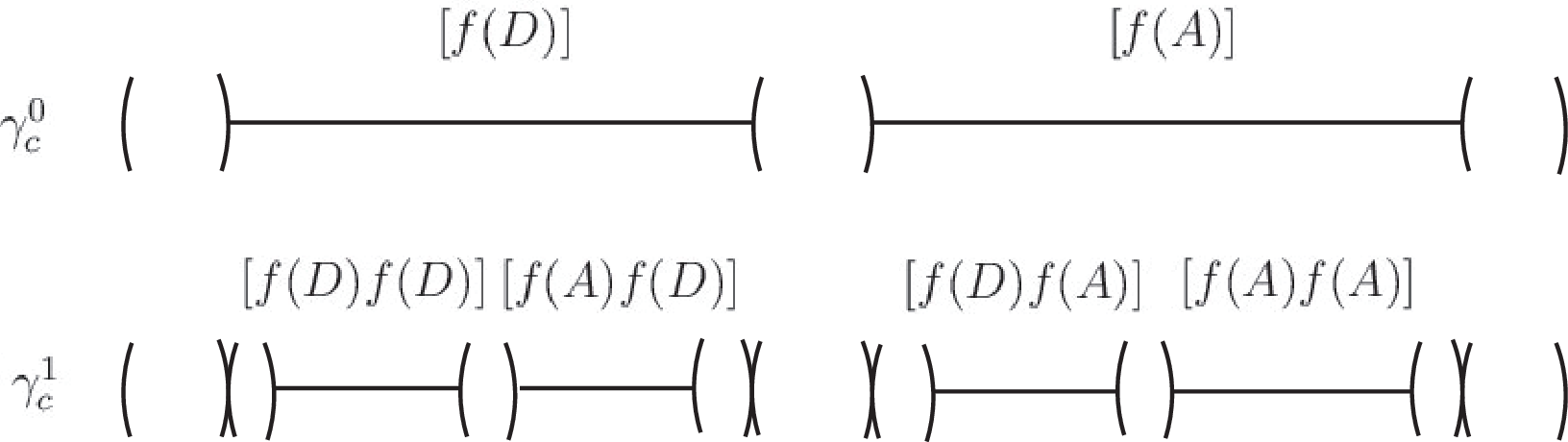}
\caption{All the $1$- and $2$-cylinders in $\gamma_c$.}
 \label{fig:cylinder-im-old}
\end{figure}

\section{Estimates on thicknesses and proofs of the main results}
In this last section we complete the proofs of Theorems~A and C.
In Section~\ref{bddis}, we establish a bounded distortion result
on iterations of curves in the graph-invariant surfaces constructed in Section~3. 
In Section~\ref{section-gap},
we estimate local thicknesses of the cross sections $\gamma_c\cap\Gamma$ and
$\sigma_c\cap\Sigma$.
In Section~\ref{lip-s}
we show that
these estimates are well-preserved under the projections $\Pi^u$ and $\Pi^s$ in \eqref{proj-us}.
In Section~\ref{section-lip} we
estimate the thicknesses of the Cantor sets $\Omega_1$ and $\Omega_2$
in \eqref{omega12}, and
complete the proofs of Theorems~A and C in Sections~\ref{pf-thma} and \ref{pf-thmb} respectively.
Computational proofs are postponed to Sections~\ref{pf1} and \ref{pf2}.

\subsection{Bounded distortions on curves in graph-invariant surfaces}\label{bddis}
Let $\theta>0$ be sufficiently small and 
let $f\in{\rm Diff}^2(\mathbb R^3)$ be sufficiently $C^2$-close to $f_0$
for which the conclusions of Propositions~\ref{tangent}, \ref{cone-cor}, \ref{invariance} hold.
A $C^2$ curve $\gamma$ in $V^{cu}\cup f(V^{cu})$
is called {\it $cu$-admissible}
if it is tangent to $\mathcal C^{c}(\theta)$
and the curvature of $\pi_1(\gamma)$ is everywhere at most
$\theta$. 
Similarly,
a $C^2$ curve $\sigma$ in $V^{cs}\cup f^{-1}(V^{cs})$
is called {\it $cs$-admissible}
if it is tangent to $\mathcal C^{c}(\theta)$
and the curvature of $\pi_2(\sigma)$ is everywhere at most
$\theta$.

\begin{prop}\label{dist-c}
  For any $K\in(1,2)$ there exists $\theta\in(0,1/10)$ 
  such that if $f\in{\rm Diff}^2(\mathbb R^3)$ is sufficiently $C^2$-close to $f_0$, then
for any $n\geq1$ and any $cu$-admissible curve $\gamma$ in $\bigcap_{i=0}^{n} f^{i}(V^{cu})$,
$f^{-n}(\gamma)$ is a $cu$-admissible curve and 
\begin{equation}\label{dist-c-eq1}\sup_{p,q\in\gamma}\frac{\|Tf^{-n}|_{T_{p}
\gamma}\|}{\|Tf^{-n}|_{T_{q}\gamma}\|}\leq K.\end{equation}
Similarly, for any $n\geq1$ and any $cs$-admissible curve $\sigma$ in 
 $\bigcap_{i=0}^{n} f^{-i}(V^{cs})$,
$f^{n}(\sigma)$ is a $cs$-admissible curve and \begin{equation}\label{dist-c-eq2}\sup_{p,q\in\sigma}\frac{\|Tf^{n}|_{T_{p}
\sigma}\|}{\|Tf^{n}|_{T_{q}\sigma}\|}\leq K.\end{equation}
\end{prop}

A proof of Proposition~\ref{dist-c} is given in Section~\ref{pf1}.\medskip

 \noindent{\bf  (Standing hypotheses for the rest of the paper)}: 
 $K\in(1,2)$ is sufficiently close to $1$, $\theta>0$ is sufficiently small,  $\lambda_u$, $\lambda_c$, $\lambda_c'$, $\lambda_s$, $\nu_u$, $\nu_c$, $\nu_c'$, $\nu_s\in\mathbb R$ satisfy \eqref{condition10}, \eqref{condition20},
and 
 $f\in{\rm Diff}^2(\mathbb R^3)$ is sufficiently $C^2$-close to $f_0$
for which the conclusions of Propositions~\ref{tangent}, \ref{cone-cor}, \ref{invariance}, \ref{dist-c} hold. In addition, $\lambda_c-\lambda_c'$ and 
 $\nu_c-\nu_c'$ are sufficiently small,
which means that $\lambda_c$, $\lambda_c'$ are sufficiently close to $\lambda_{0,c}$ and $\nu_c$, $\nu_c'$ are sufficiently close to $\nu_{0,c}$.

\subsection{Lower bounds on local thicknesses of the cross sections}\label{section-gap}
As described in Section~\ref{s-c}, we will consider the backward iteration of $\gamma_c$ and the forward iteration of $\sigma_c$.
Let $\ell$ denote the arc-length measure on $L\cup\bigcup_{n=0}^\infty f^{-n}(\gamma_c)
\cup\bigcup_{n=0}^\infty f^{n}(\sigma_c)$ with respect to the induced metric.
\begin{prop}\label{core}
If $\delta>0$ and $f\in{\rm Diff}^2(\mathbb R^3)$ is 
sufficiently $C^2$-close to $f_0$,
then for any $x\in \gamma_c\cap\Gamma$ which is a boundary point of a bounded gap of $\gamma_c\cap\Gamma$,
\[\tau(\gamma_c\cap\Gamma,x)> a_1-\delta.\]
Similarly, for any $y\in \sigma_c\cap\Sigma$ which is a boundary point of a bounded gap of $\sigma_c\cap\Sigma$,
\[\tau(\sigma_c\cap\Sigma,y)> a_2-\delta.\]
\end{prop}

\begin{figure}
\centering
\includegraphics[height=0.5\textwidth,width=0.7\textwidth]
{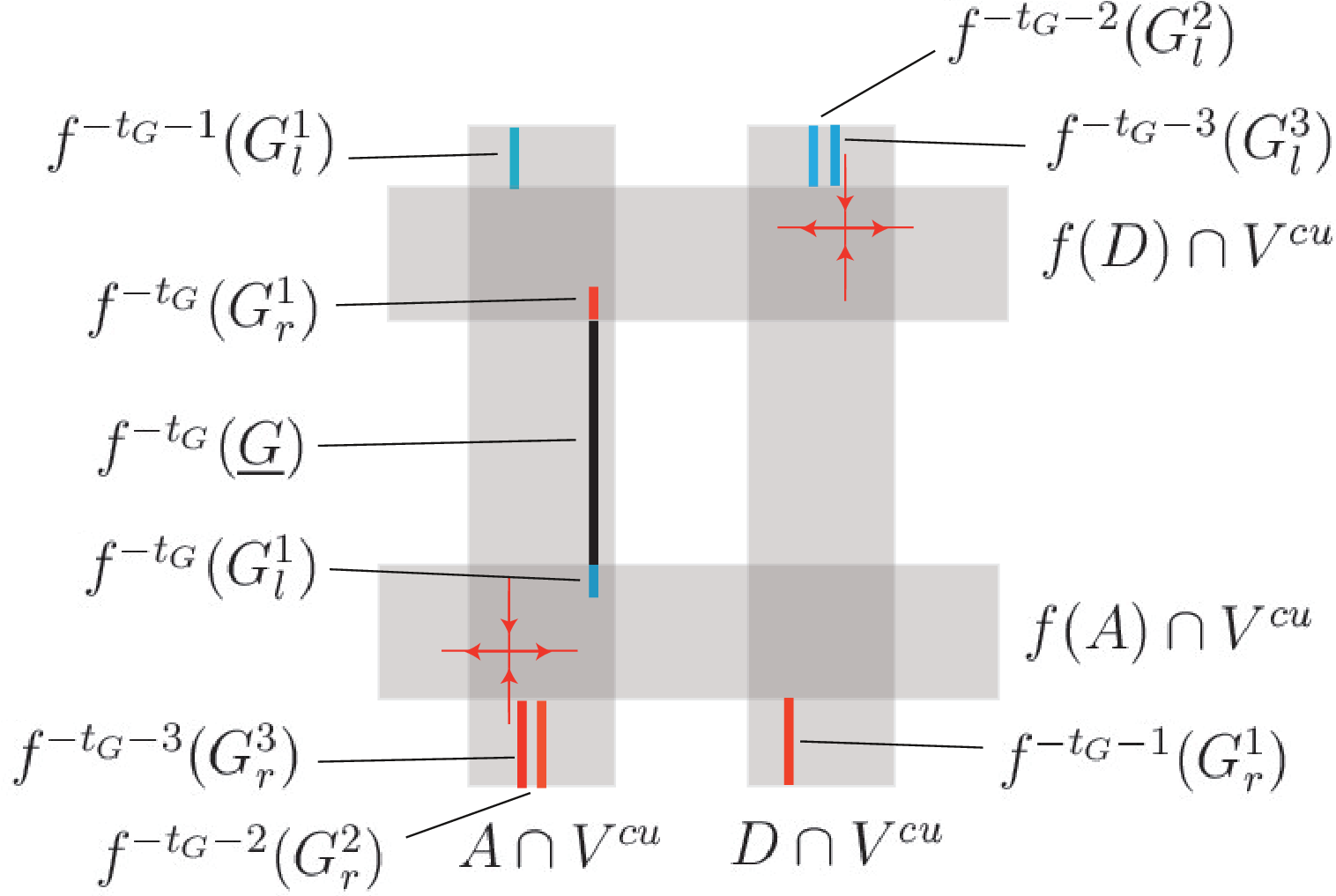}
\caption{Locations of $f^{-t_G-n}(G_r^n)$, $f^{-t_G-n}(G_l^n)$ $(0\leq n\leq3)$ and 
the fixed saddles in $A$, $D$.}
 \label{fig:gap-im}
\end{figure}

\begin{proof}
 Since
$\gamma_c$ is a transverse intersection of two $C^2$ surfaces whose derivatives up to the second order can be made arbitrarily small by taking $f$ sufficiently $C^2$-close to $f_0$, we may assume $\gamma_c$ is a $cu$-admissible curve. For the same reason, we may assume $\sigma_c$ is a $cs$-admissible curve.

To proceed, we need two lemmas on lengths of gaps.
\begin{lemma}\label{gap-com}
 For any bounded gap $G$ of $\gamma_c\cap\Gamma$ we have
\[ \left|\frac{\ell(f^{-t_G}(\underline{G}))}{|E_c^*|-
2|A_c^*|}-1\right|
   \leq \theta\quad
   \text{and}\quad \left|\frac{\ell(f^{-t_G}(G\setminus\underline{G}))}{\kappa_1|A_c^*| }-1\right|\leq 2\theta.\]
  Similarly, for
any bounded gap $H$ of $\sigma_c\cap\Sigma$ we have
   \[
   \left| \frac{\ell(f^{u_H}(\underline{H}))}{|F_c|-2|B_c|}-1\right|
   \leq \theta\quad\text{and}\quad \left|\frac{\ell(f^{u_H}(H\setminus\underline{H}))}{\kappa_2|B_c| }-1\right|\leq 2\theta.\]
   \end{lemma}
   \begin{proof}
Let $G$ be a bounded gap of $\gamma_c\cap\Gamma.$
Since $\gamma_c$ is tangent to
$\mathcal C^{c}(\theta)$, so is
$G$. From Proposition~\ref{tangent}, the curve
$f^{-t_G}(\underline{G})$ is tangent to $\mathcal C^{c}(\theta)$, and
 joins two points in the boundaries of the two connected components of $f(V^{cu})$.
 This implies the first inequality in Lemma~\ref{gap-com}.

For each $n\geq1$, we define 
\[G_r^n=[f(A)^{n-1}f(D)\omega^-(G)]\setminus[f(A)^{n}f(D)\omega^-(G)]\cup[f(D)f(A)^{n-1}f(D)\omega^-(G)],\]
and
\[G_l^n=[f(D)^{n-1}f(A)\omega^-(G)]\setminus
([f(A)f(D)^{n-1}f(A)\omega^-(G)]
\cup [f(D)^{n}f(A)\omega^-(G)]).\]
We have
\[G\setminus\underline{G}=\bigcup_{n=1}^\infty (G_r^n
\cup  G_l^n).\]
From Proposition~\ref{tangent}, $f^{-i}(G_r^n)$,
$f^{-i}(G_l^n)$ are
tangent to $\mathcal C^{c}(\theta)$ for $0\leq i\leq t_G+n$, and
$f^{-t_G-n}(G_r^n)$,
$f^{-t_G-n}(G^n_l)$ stretch across the connected components of $V^{cu}\setminus f(V^{cu})$ intersecting $A_u\times\partial A_c\times A_s$ or  $D_u\times\partial D_c\times D_s$ as in
FIGURE~\ref{fig:gap-im}. This implies
\[\left|\frac{\ell(f^{-t_G-n}(G_r^n\cup G^n_l))}{|A_c|-|E_c^*|}-1\right|\leq\theta.\]
From \eqref{cone-eq2} in Proposition~\ref{cone-cor}(a),
we have
\[{\lambda_c'}^n\leq \frac{\ell(f^{-t_G}(G_r^n))}{\ell(f^{-t_G-n}(G_r^n))}\leq\lambda_c^n\quad\text{and}\quad{\lambda_c'}^n\leq \frac{\ell(f^{-t_G}(G_l^n))}{\ell(f^{-t_G-n}(G_l^n))}\leq\lambda_c^n.\]
Combining these inequalities yields
\begin{equation}\label{s4g}(1-\theta){\lambda_{c}'}^{n}\leq\frac{\ell(f^{-t_G}(G_r^n\cup G^n_l))}{|A_c|-|E_c^*|}\leq(1+\theta)\lambda_c^n.\end{equation}
If $\lambda_c-\lambda_c'$ is sufficiently small, then
summing \eqref{s4g} over all $n\geq1$ 
and combining the result with $\lambda_{0,c}/(1-\lambda_{0,c})
=|A_c^*|/(|A_c|-|A_c^*|)$,
we obtain the second inequality in 
Lemma~\ref{gap-com}. Proofs of the remaining two are analogous.
\end{proof}

\begin{lemma}\label{thick-new}
For any pair $G$, $G'$ of different bounded gaps  
of $\gamma_c\cap\Gamma$
such that 
$[\omega^-(G)]\supset G'$, we have
$\ell(G)>\ell(G')$. Similarly,
for any pair $H$, $H'$ of different bounded gaps 
of $\sigma_c\cap\Sigma$
such that 
$[\omega^+(H)]\supset H'$, we have
$\ell(H)>\ell(H').$
\end{lemma}
\begin{proof}
Let $G$, $G'$ be bounded gaps of $\gamma_c\cap\Gamma$ as in the first assertion of the lemma. 
 We have $t_{G}<t_{G'}$.
Using \eqref{dist-c-eq1} in Proposition~\ref{dist-c} and
$\ell(f^{-t_{G}}(G'))\leq \lambda_{c}^{t_{G'}-t_G}\ell(f^{-t_{G'}}(G'))$ which follows from \eqref{cone-eq2} in Proposition~\ref{cone-cor}(a), we have
\[\frac{\ell(G)}{\ell(G')}\geq K^{-1}\frac{\ell(f^{-t_{G}}(G))}{\ell(f^{-t_{G}}(G'))}\geq K^{-1}
\lambda_{c}^{-t_{G'}+t_{G}}\frac{\ell(f^{-t_{G}}(G))}{\ell(f^{-t_{G'}}(G'))}.\]
Applying Lemma~\ref{gap-com} 
to the last fraction we 
  obtain
\[\frac{\ell(G)}{\ell(G')}\geq\frac{K^{-1}}{\lambda_{c} }\frac{1-\theta}{1+\theta}\frac{|E_c^*|-2|A_c^*| +   \kappa_1|A_c^*| }{|E_c^*|-2|A_c^*|+ \kappa_1|A_c^*| }.\]
Since $K^{-1}\lambda_c^{-1}>1$,
the right-hand side is strictly larger than $1$.
  We have verified the first assertion of Lemma~\ref{thick-new}. A proof of the second one is analogous.
\end{proof}

Returning to the proof of Proposition~\ref{core},
we only give a proof of the first inequality 
 since that of the second one is analogous.
Let $G$ be a bounded gap of $\gamma_c\cap\Gamma$ and let $x\in\partial G$.
With no loss of generality we may assume $f^{-t_G}(x)\in f(A)$.
Let $I$ denote the bridge of $\gamma_c\cap\Gamma$
at $x$. Let $\tilde I$ denote the connected component of 
$[\omega^-(G)]\setminus([ f(A)\omega^-(G)]\cup [ f(D)\omega^-(G)])$ that contains $x$.
Let $I'$ denote the minimal curve in $\gamma_c$ that contains $x$ and all 
gaps of $\gamma_c\cap\Gamma$ contained in $[f(A)\omega^-(G)]$.
Clearly we have $I'\subset \tilde I$. Since $f^{-t_G}(\tilde I)$ is tangent to $\mathcal C^c(\theta)$, we have
\begin{equation}\label{eq101}\left|\frac{\ell(f^{-t_G}(\tilde I))}{|A_c^*|} -1\right|\leq\theta.\end{equation}
A similar reasoning to the proof of Lemma~\ref{gap-com} shows that
\begin{equation}\label{eq102}\frac{\ell(f^{-t_G}(\tilde I\setminus I'))}{\kappa_1|A_c^*|}\leq 1+2\theta.\end{equation}
Using \eqref{eq101}, \eqref{eq102} and $I'\subset I$ which follows from
Lemma~\ref{thick-new} and the definition of bridge in Section~\ref{thick}, we obtain 
\[\begin{split}\frac{\ell(I)}{\ell(G)}\geq\frac{\ell(I')}{\ell(G)}\geq K^{-1}
\frac{\ell(f^{-t_G}(I'))}{\ell(f^{-t_G}(G))}\geq K^{-1}
\frac{(1-\theta)|A_c^*|-(1+2\theta)\kappa_1 |A_c^*|}{(1+\theta)(|E_c^*|-2|A_c^*|)+(1+2\theta) \kappa_1|A_c^*| }.
\end{split}\]
If $K-1$ and $\theta$ are sufficiently small, the first inequality in Proposition~\ref{core} holds.
\end{proof}

\subsection{Lipschitz continuity of holonomy maps}\label{lip-s}
For a pair $\gamma$, $\gamma'$ of $cu$-admissible curves, 
we write $\gamma\sim\gamma'$ if
$\Pi^u(\gamma\cap \mathcal F^u(\Gamma))=\Pi^u(\gamma'\cap \mathcal F^u(\Gamma))$.
If $\gamma\sim\gamma'$,
define $\Pi^u_{\gamma\gamma'}\colon
 \gamma\cap \mathcal F^u(\Gamma)\to \gamma'\cap \mathcal F^u(\Gamma)$
 by $\Pi^u_{\gamma\gamma'}(p)\in\gamma'\cap\mathcal F^u(p)$.
 Note that $\Pi^u_{\gamma\gamma'}$
 is invertible and the inverse is 
 $\Pi^u_{\gamma'\gamma}$.
Similarly, for a pair $\sigma$, $\sigma'$ of $cs$-admissible curves,
   we write $\sigma\sim\sigma'$
if $\Pi^s(\sigma\cap \mathcal F^s(\Sigma))=\Pi^s(\sigma'\cap \mathcal F^s(\Sigma))$. If $\sigma\sim\sigma'$,
define $\Pi^s_{\sigma\sigma'}\colon
 \sigma\cap \mathcal F^s(\Sigma)\to \sigma'\cap \mathcal F^s(\Sigma)$
 by $\Pi^s_{\sigma\sigma'}(q)\in\sigma'\cap\mathcal F^s(q)$.

The next proposition asserts that
the above maps
are bi-Lipschitz continuous, and the Lipschitz constants can be made arbitrarily close to $1$ by taking $f$ which is sufficiently $C^2$-close to $f_0$. 

\begin{prop}\label{lip}
  For any $K\in(1,2)$ there exists $\theta\in(0,1/10)$ such that if $f\in{\rm Diff}^2(\mathbb R^3)$ is sufficiently $C^2$-close to $f_0$, then for any pair  $\gamma$, $\gamma'$ of
 $cu$-admissible curves with
$\gamma\sim\gamma'$ we have
    \[\sup_{\stackrel{p,q\in \gamma\cap \mathcal F^u(\Gamma)}{p\neq q}}
    \frac{d(\Pi^u_{\gamma\gamma'}(p),\Pi^u_{\gamma\gamma'}(q))}{d(p,q)}\leq K^4.\]
 Similarly, for any pair $\sigma$, $\sigma'$ of two $cs$-admissible curves with
$\sigma\sim\sigma'$, 
  \[\sup_{\stackrel{p,q\in \gamma\cap \mathcal F^s(\Sigma)}{p\neq q}}
    \frac{d(\Pi^s_{\sigma\sigma'}(p),\Pi^s_{\sigma\sigma'}(q))}{d(p,q)}\leq K^4.\]
\end{prop}
A proof of Proposition~\ref{lip} is given in Section~\ref{pf2}.

\subsection{Lower bounds on thicknesses the Cantor sets}\label{section-lip}
We now obtain the following lower bounds on the thicknesses of $\Omega_1$
and $\Omega_2$ in \eqref{omega12}.
\begin{prop}\label{as-eq}
If $\delta>0$ and $f\in{\rm Diff}^2(\mathbb R^3)$ is 
sufficiently $C^2$-close to $f_0$,
then
\[\tau(\Omega_1)\geq a_1-\delta\quad\text{and}\quad
\tau(\Omega_2)\geq a_2-\delta.\]
\end{prop}
\begin{proof}
Let $G$ be a bounded gap of $\Omega_1$ 
and let $x\in\partial G$. Let $I$ denote the bridge of $\Omega_1$  at $x$.  
 Let $G'$ denote the bounded gap of $\gamma_c\cap\Gamma$ such that  $\partial G=\Pi^u(\partial G')$.
 Let $I'$
 denote the bridge of $\gamma_c\cap \Gamma$ at $y\in\partial G'$ such that $\Pi^u(y)=x$. By
Propositions~\ref{core} and \ref{lip},
\[\tau(\Omega_1,x)=\frac{\ell(I)}{\ell(G)}\geq K^{-8}\frac{1}{(1+\theta)^2}
\frac{\ell( I')}{\ell(G')}\geq K^{-8}\frac{1}{(1+\theta)^2}\left(a_1-\frac{\delta}{2}\right)>a_1-\delta.\]
Since $G$ is an arbitrary bounded gap of $\Omega_1$ and $x$ is an arbitrary point of $\partial G$, we obtain
$\tau(\Omega_1)\geq a_1-\delta$.
A proof of the second inequality is analogous. 
\end{proof}

\subsection{Proof of Theorem~A}\label{pf-thma}
From Proposition~\ref{as-eq} and Lemma~\ref{gaplem}
together with the assumption $a_1a_2>1$,
 $\mathscr{H}$ is non-empty.
The upper bounds on $\dim_{\rm H} W^u(\Gamma)$ and $\dim_{\rm H} W^s(\Sigma)$ follow from Proposition~\ref{lip}.
The lower bounds are consequences
of Propositions~\ref{lip}, \ref{as-eq} and the result of Newhouse
 \cite{New79} (see also \cite[p.77, Proposition~5]{PalTak93}) which asserts that
the Hausdorff dimension of a Cantor set in $\mathbb R$ with thickness $\tau$ is at least $\log2/\log(2+1/\tau)$.

We claim that
there exist countably many $C^1$ curves $\gamma_1,\gamma_2,\ldots$ in
$V^{cu}\cap\bigcup_{n=0}^{\infty} f^{-n}(V^{cs})$ which are
 tangent to $\mathcal C^{c}(\theta)$
and satisfy
$\mathscr{H}\subset \bigcup_{k=1}^\infty\bigcup_{n=1}^{\infty}f^{n}(\gamma_k\cap 
\mathcal F^u(\Gamma))$.
Propositions~\ref{tangent} and \ref{cone-cor} together imply that
 each $\gamma_k\cap \mathcal F^u(\Gamma)$ is covered by 
$2^n$ curves in $\gamma_k$ of length at most $\lambda_c^{n}$, for all $n\geq1$, and hence
$\dim_{\rm H}(\gamma_k\cap \mathcal F^u(\Gamma))\leq-
\log2/\log\lambda_c<1$.
By the countable stability 
and the invariance of Hausdorff dimension under 
bi-Lipschitz homeomorphisms, we obtain
$\dim_{\rm H} \mathscr{H}\leq-
\log2/\log
\lambda_c<1$. 
In particular, $\mathscr{H}$ is totally disconnected \cite[Proposition~2.5]{F}. 

It is left to prove the claim.
For each $p\in \mathscr{H}$,
there are integers $n_0$, $n_1$ with $n_0< 0\leq n_1$ such that $f^{n}(p)\in V^{cu}\cap \mathcal F^u(\Gamma)$ for $n\leq n_0$
and $f^{n}(p)\in V^{cs}\cap \mathcal \mathcal F^s(\Sigma)$ for $n\geq n_1$.
Set $N=n_1-n_0$.
By Proposition~\ref{cone-cor}, 
$f^{N}(V^{cu})$ intersects $V^{cs}$ transversely at $f^{n_1}(p)$,
and so $f^{N}(V^{cu})\cap V^{cs}$ contains a curve which is tangent to  $\mathcal C^{c}(\theta)$ and contains
$f^{n_1}(p)$.
  We apply Proposition~\ref{cone-cor} to $f^{-N}$ to obtain a curve
in
$V^{cu}\cap f^{-N}(V^{cs})$ which is tangent to $\mathcal C^{c}(\theta)$
and contains $f^{n_0}(p)$.
The proof of the claim is complete.\qed

\subsection{Proof of Theorem~C}\label{pf-thmb}
In view of Theorem~\ref{cant-inter},
 we fix $T_0>1$ such that if $S_1$, $S_2$
are two interleaved Cantor sets in $\mathbb R$ with $\tau(S_1)> T_0/2$, $\tau(S_2)> T_0/2$ then $S_1\cap S_2$ contains a Cantor set whose thickness is at least
$(6/7)\sqrt{\min\{\tau(S_1),\tau(S_2)\}}$.
Assume $\min\{a_1,a_2\}>T_0$.
By Proposition~\ref{as-eq},
we obtain $\tau(\Omega_1)>T_0/2$ and $\tau(\Omega_2)>T_0/2$
for all $f$ that is sufficiently $C^2$-close to $f_0$.
Then \[\tau(\Omega_1\cap \Omega_2)\geq\frac{6}{7}\sqrt{\min\{\tau(\Omega_1),\tau(\Omega_2)\}}>\frac{5}{6}\sqrt{\min\{a_1,a_2\}}.\]
 Using Newhouse's lower bound as in the proof of Theorem~A, we obtain
\[\begin{split}
\frac{1}{\dim_{\rm H} \mathscr{H}}-1&<
\frac{1}{2\log2}\frac{1}{\tau(\Omega_1\cap \Omega_2)}
<
\frac{1}{\sqrt{\min\{a_1,a_2\}}},
\end{split}\]
as required.\qed

\subsection{Proof of Proposition~\ref{dist-c}}\label{pf1}
We consider iterations of planar
 maps $\varphi_1\colon \pi_1(R_1)\to\pi_1(R_1)$
 and $\varphi_2\colon \pi_2(R_2)\to\pi_2(R_2)$ given by
\begin{equation}\label{varphieq}\varphi_1=\pi_1\circ f^{-1}\circ(\pi_1|_{V^{cu}})^{-1}
\quad\text{and}\quad
\varphi_2=\pi_2\circ f\circ(\pi_2|_{V^{cs}})^{-1}.\end{equation}
Recell that
$V^{cu}$, $V^{cs}$ are 
$C^2$ surfaces, and so
  $\varphi_1$, $\varphi_2$ are of class $C^2$. 
Let $J\varphi_1$, $J\varphi_2$ denote the Jacobian matrices of $\varphi_1$, $\varphi_2$
with respect to the canonical basis. Write
\[J\varphi_1=\begin{pmatrix} \Phi_{uu} & \Phi_{uc}\\\Phi_{cu}&\Phi_{cc} \end{pmatrix} \quad\text{and}\quad J\varphi_2=\begin{pmatrix} \Phi_{cc} & \Phi_{cs}\\\Phi_{sc}&\Phi_{ss} \end{pmatrix},\]
and put
\[J^2\varphi_1=\begin{pmatrix} \nabla_1\Phi_{uu} & \nabla_1\Phi_{uc}\\ \nabla_1\Phi_{cu}&\nabla_1\Phi_{cc} \end{pmatrix} \quad\text{and}\quad J^2\varphi_2=\begin{pmatrix} \nabla_2\Phi_{cc} & \nabla_2\Phi_{cs}\\ \nabla_2\Phi_{sc}&\nabla_2\Phi_{ss} \end{pmatrix},\]
where
$\nabla_1=(\partial_{x_u},\partial_{x_c})$
and $\nabla_2=(\partial_{x_c},\partial_{x_s})$.
Replacing $f$ if necessary
we may assume
\begin{equation}\label{deriv}
\sup_{\pi_1(R_1)}\|J\varphi_1\|
<(1+\theta)\lambda_{0,c}^{-1}
\quad\text{and}\quad
\sup_{\pi_2(R_2)}\| J\varphi_2\|<(1+\theta)\nu_{0,c},\end{equation}
\begin{equation}\label{g-der}
\sup_{\pi_1(R_1)}|\det J\varphi_1|
<(1+\theta)\lambda_{0,u}^{-1}\lambda_{0,c}^{-1}
\quad\text{and}\quad
\sup_{\pi_2(R_2)}|\det J\varphi_2|<(1+\theta)\nu_{0,c}\nu_{0,s},\end{equation}
\begin{equation}\label{2-del}
    \sup_{\pi_1(R_1)}\|J^2\varphi_1\|\leq \frac{\theta\lambda_{0,c}}{1+\theta}\quad\text{and}\quad\sup_{\pi_2(R_2)}\|J^2\varphi_2\|\leq\frac{\theta\nu_{0,c}^{-1}}{1+\theta}.\end{equation}

Let $n\geq1$ and let $\gamma$ be a $cu$-admissible curve as in Proposition~\ref{dist-c}.
From the assumption $\gamma\subset\bigcap_{i=0}^{n} f^{i}(V^{cu})$ and Proposition~\ref{tangent}, the curve
$f^{-i}(\gamma)$ is tangent 
to $\mathcal C^{c}(\theta)$ for $0\leq i\leq n$.
To estimate curvatures, we parametrize 
$\gamma_0=\pi_1(\gamma)$ by arc length $s$
and put $\gamma_{i+1}(s)=\varphi_1\circ\gamma_i(s)$ inductively for $i=0,\ldots,n-1$. 
Then
\[\label{curv-form}\dot{\gamma}_{i+1}(s)=M_i(s)\dot\gamma_{i}(s)\quad
\text{and}\quad
\ddot{\gamma}_{i+1}(s)=N_i(s)\dot\gamma_{i}(s)+
M_i(s)\ddot\gamma_{i}(s),\]
where
 \[M_i(s)=J\varphi_1(\gamma_{i}(s))\quad\text{and}\quad N_i(s)=\begin{pmatrix}\langle \nabla \Phi_{uu},\dot{\gamma}_{i}(s) \rangle & \langle \nabla \Phi_{uc},\dot{\gamma}_{i}(s) \rangle\\
\langle \nabla \Phi_{cu},\dot{\gamma}_{i}(s) \rangle& \langle \nabla \Phi_{cc},\dot{\gamma}_{i}(s) \rangle\end{pmatrix},\]
and the single and double dots denote the first- and the second-order derivatives on $s$ respectively, and 
$\langle\cdot,\cdot\rangle$ denotes the inner product.
Let $\kappa_{i}(s)$ denote the curvature of $\gamma_{i}$ at $\gamma_{i}(s)$.
We have
\begin{equation}\label{curvature-ineq}
\begin{split}
    \kappa_{i+1}(s)&
    \leq\frac{\|M_{i}(s)\dot\gamma_{i}(s)\times
    N_i(s)\dot\gamma_{i}(s)\|+|\det M_i(s)|\cdot\|\dot{\gamma}_i(s)\times\ddot{\gamma}_i(s)\|}
    {\|\dot{\gamma}_{i+1}(s)\|^3}\\
    &\leq\frac{\|\dot{\gamma}_{i}(s)\|^3}{\|\dot{\gamma}_{i+1}(s)\|^3}\left(
\|M_{i}(s) \|
    \|N_i(s)\|\frac{1}{\|\dot{\gamma}_{i}(s)\|}+
    |\det M_i(s)|\kappa_i(s)\right)\\
    &\leq\frac{\theta}{8}+\frac{1+\theta}{8}\kappa_i(s).
\end{split}
\end{equation}
By virtue of \eqref{deriv}, \eqref{g-der}, \eqref{2-del},
the last inequality in \eqref{curvature-ineq} holds if $\lambda_{c}-\lambda_c'$ is sufficiently small.
  Using \eqref{curvature-ineq} inductively and combining the result with $\kappa_0(s)\leq\theta$ yields
\begin{equation}\label{kappan}\kappa_{n}(s)\leq \frac{\theta}{8}\sum_{i=0}^{n-1}
\left(\frac{1+\theta}{8}\right)^{i}+\left(\frac{1+\theta}{8}\right)^{n}\kappa_0(s)<
\theta.\end{equation}
Since $\pi_1(f^{-n}(\gamma(s)))=\gamma_n(s)$, from \eqref{kappan} it follows that 
$f^{-n}(\gamma)$ is a $cu$-admissible curve as required in Proposition~\ref{dist-c}.

Let $0\leq i\leq n-1$.
For all parameter values $s$, $t$ we have
\[
    \begin{split}
\left|\frac{\|\dot{\gamma}_{i+1}(s)\|}{\|\dot{\gamma}_{i}(s)\|}-
\frac{\|\dot{\gamma}_{i+1}(t)\|}{\|\dot{\gamma}_{i}(t)\|}\right|
&\leq\theta\|M_i(t)\| |\gamma_i(s)-\gamma_i(t)|+\|M_i(s)-M_i(t)\|\\
&\leq\sqrt{\theta}|\gamma_i(s)-\gamma_i(t)|\leq 2\sqrt{\theta}{\lambda_c'}^{n-i},
\end{split}
\]
where $|\gamma_i(s)-\gamma_i(t)|$ denotes the Euclidean distance between $\gamma_i(s)$
and $\gamma_i(t)$.
We have used \eqref{kappan} for the first inequality, 
\eqref{g-der} and \eqref{2-del} for the second one.
The last inequality follows from \eqref{cone-eq2} in Proposition~\ref{cone-cor}.
If $\theta<((1/4)\log K)^2$ then
\[\begin{split} \log\frac{\|\dot{\gamma}_{n}(s)\|}{\|\dot{\gamma}_{n}(t)\|}\leq 
    \sum_{i=0}^{n-1}\left|\frac{\|\dot{\gamma}_{i+1}(s)\|}{\|\dot{\gamma}_{i}(s)\|}-
\frac{\|\dot{\gamma}_{i+1}(t)\|}{\|\dot{\gamma}_{i}(t)\|}\right|
<  2\sqrt{\theta}<\frac{1}{2}\log K.\end{split}\]
Since
$f^{-n}(\gamma)=(\pi_1|_{V^{cu}})^{-1} \circ \varphi_1^n\circ\pi_1(\gamma)$ and
 the curves $\gamma$, $f^{-n}(\gamma)$ are tangent to $\mathcal C^{cu}(\theta)$, we obtain
 \eqref{dist-c-eq1}
  if
 $\theta$ is sufficiently small.
 Exchanging the roles of
$\varphi_1$, $\varphi_2$ and proceeding in the same way we obtain
\eqref{dist-c-eq2}.\qed

\subsection{Proof of Proposition~\ref{lip}}\label{pf2}
Let $f$ be sufficiently $C^2$-close to $f_0$ such that
for any pair $\gamma$, $\gamma'$ of $cu$-admissible curves 
with $\gamma\sim\gamma'$,
\begin{equation}\label{pf2-eq}
\sup_{\stackrel{p,q\in\gamma\cap \mathcal F^u(\Gamma)}{p\neq q}}\frac{d(f^{-n(p,q)}(\Pi^u_{\gamma\gamma'}(p)),f^{-n(p,q)}(\Pi^u_{\gamma\gamma'}( q)))}{d(f^{-n(p,q)}(p),f^{-n(p,q)}(q))}\leq K.\end{equation}
Here, $n(p,q)$ denotes the non-negative minimal integer such that $f^{-n(p,q)}(p)$ and 
$f^{-n(p,q)}(q)$ are not contained in the same connected component of $f(V^{cu})$.
The first inclusion in \eqref{include-f} implies that $n(p,q)$ makes sense.

 Let $\gamma$, $\gamma'$ be $cu$-admissible curves such that
$\gamma\sim\gamma'$.
Let $p,q\in\gamma\cap \mathcal F^u(\Gamma)$, $p\neq q$ and put $n=n(p,q)$.
Note that $\gamma\cup\gamma'\in\bigcap_{k=0}^n f^k(R_1)$.
From \eqref{pf2-eq} and the mean value theorem, there exist $\xi\in\gamma$ 
and $\eta\in\gamma'$ 
such that
\[\frac{d(\Pi^u_{\gamma\gamma'}(p),\Pi^u_{\gamma\gamma'}( q))}
{d(p,q)}
\leq K
\frac{\|Tf^{-n}|_{T_\xi\gamma}\|}{\|Tf^{-n}|_{
T_\eta\gamma'}\|}.\]
Put $p'= 
\Pi^u_{\gamma\gamma'}(p)$. 
The chain rule and \eqref{dist-c-eq1} together imply
  \[\frac{\|Tf^{-n}|_{T_\xi\gamma}\|}{\|Tf^{-n}|_{
T_\eta\gamma'}\|}=\frac{\|Tf^{-n}|_{T_\xi\gamma}\|}{\|Tf^{-n}|_{
T_p\gamma}\|}\frac{\|Tf^{-n}|_{T_p\gamma}\|}{\|Tf^{-n}|_{
T_{p'}\gamma'}\|}\frac{\|Tf^{-n}|_{T_{p'}\gamma'}\|}{\|Tf^{-n}|_{
T_\eta\gamma'}\|}
\leq K^3\frac{\|Tf^{-n}|_{T_p\gamma}\|}{\|Tf^{-n}|_{
T_{p'}\gamma'}\|}.\]
Hence,
for the proof of the first inequality in Proposition~\ref{lip} it suffices to show
\begin{equation}\label{dist-c'-eq2}\frac{\|Tf^{-n}|_{T_p\gamma}\|}{\|Tf^{-n}|_{
T_{p'}\gamma'}\|}\leq
K.\end{equation}

To show \eqref{dist-c'-eq2},
we use the map $\varphi_1$ in \eqref{varphieq}.
Let $v_0$, $v_0'$ be unit vectors at $\pi_1(p_0)$ and $\pi_1(p_0')$ which are tangent to $\pi_1(\gamma)$ and $\pi_1(\gamma')$ respectively. 
Put
$M_i=J_{\pi_1(p_{-i}) }\varphi_1$, $M_i'=J_{\pi_1(p_{-i}') }\varphi_1$
and $v_{i+1}=M_iv_{i}$, $v_i'=M_i'v_{i}'$ inductively for $i=0,\ldots,n-1$,
 where $p_{-i}=f^{-i}(p)$ and $p_{-i}'=f^{-i}(p')$.
Put $s_i=\arccos\langle v_i,v_i'\rangle$.
We set $c_0=\lambda_c^2\lambda_{0,c}^{-1}$ and
 assume $\lambda_c-\lambda_c'$ is sufficiently small so that $c_0<1$.

For $0\leq i\leq n-1$ we have
\begin{equation}\label{si-eq}
\left|\frac{\|v_{i+1}\|}{\|v_{i}\|}-\frac{\|v_{i+1}'\|}{\|v_{i}'\|}\right|\leq 
\theta |p_{-i}-p_{-i}'|
+\|M_i'\|s_i\leq  \frac{\theta}{2^{i-1}}+(1+\theta)\lambda_{0,c}^{-1}s_i,\end{equation}
and if $i\geq1$ then
\begin{equation}\label{si-eq2}
\begin{split}
s_i&=
\frac{\|M_{i-1}v_{i-1}\times M_{i-1}v_{i-1}'+M_{i-1}v_{i-1}\times(M_{i-1}'-M_{i-1} )v_{i-1}'\|}{\|v_{i}\|\|v_{i}'\|}
\\
&\leq\frac{\|v_{i-1}\|}{\|v_{i}\|}\frac{\|v_{i-1}'\|}{\|v_{i}'\|}
(|\det M_{i-1} |s_{i-1}+ \theta 
|p_{-i+1}-p_{-i+1}'|)\\
&\leq \frac{\|v_{i-1}\|}{\|v_{i}\|}\frac{\|v_{i-1}'\|}{\|v_{i}'\|}((1+\theta)\lambda_{0,u}^{-1}\lambda_{0,c}^{-1}s_{i-1}+2\theta\lambda_u^{-i+1})\\
&\leq c_0(1+\theta)\lambda_{0,u}^{-1}s_{i-1}+\frac{\theta}{2}\lambda_u^{-i+1}\leq\frac{1+\theta}{2}s_{i-1}+\frac{\theta}{2^{i}}.
\end{split}
\end{equation}
We have used
\eqref{2-del}
to estimate the second cross product in \eqref{si-eq2}. We have used
\eqref{g-der}, and
\eqref{cone-eq1}, \eqref{cone-eq2}
in Proposition~\ref{cone-cor} to deduce 
$|p_{-i+1}-p_{-i+1}'|\leq2\lambda_u^{-i+1}$,
$\|v_{i}\|\geq2\|v_{i-1}\|$ and
$\|v_{i}'\|\geq2\|v_{i-1}'\|$ for
 the second and third inequalities in \eqref{si-eq2}.
Using \eqref{si-eq2} inductively and then
 $s_0\leq 2\theta$
yields
\begin{equation}\label{si}s_{i}\leq 2\theta \left(\frac{1+\theta}{2}\right)^{i}+\frac{\theta}{2^{i}}\sum_{k=0}^{i-1}\left(\frac{1+\theta}{2}\right)^k\leq2\theta \left(\frac{1+\theta}{2}\right)^{i}+\frac{3\theta}{2^{i}},\end{equation}
for $0\leq i\leq n-1$.
Plugging \eqref{si} into \eqref{si-eq} and 
summing the result yields
\[\log\frac{\|v_{n}\|}{\|v'_{n}\|}
\leq  \sum_{i=0}^{n-1}\left|\frac{\|v_{i+1}\|}{\|v_{i}\|}-\frac{\|v_{i+1}'\|}{\|v_i'\|}\right|
\leq \frac{1}{2}\log K.\]
Since
$f^{-n}(\gamma)=(\pi_1|_{V^{cu}})^{-1} \circ \varphi_1^n\circ\pi_1(\gamma)$, the curves
$\gamma$, $f^{-n}(\gamma)$ are tangent to $\mathcal C^{cu}(\theta)$, and the same for $\gamma'$, we obtain
\eqref{dist-c'-eq2}.
Exchanging the roles of $\varphi_1$, $\varphi_2$ and proceeding in the same way,
we obtain the second inequality in Proposition~\ref{lip}. \qed

\section*{Appendix}
Under additional conditions on the heterochaos horseshoe map $f_0$, one can obtain $C^1$-robust heterodimensional cycles. The following was pointed to us by Asaoka
in a personal communication.

Let $f_0\in{\rm Diff}^1(\mathbb R^3)$ be a 
heterochaos horseshoe map as in Section~\ref{cou-sec}.
We assume:
\begin{itemize}
\item[(i)] $(|A_c^*|/|A_c|)(|B_c|/|B_c^*|)<1$.
\item[(ii)] $A_c=D_c \subset {\rm int} (f_0^2(A\cap f_0^{-1}(B))_c) \cup {\rm int}(f_0^2(D\cap f_0^{-1}(C))_c)$,
where ${\rm int}(\cdot)$ denotes the interior operation.
\end{itemize}
The restriction of $f_0^2$ to 
$(A\cap f_0^{-1}(B))\cup (D\cap f_0^{-1}(C))$
is essentially the same as the affine model generating a blender explained in
\cite[Section~6.2.1]{BDV04} (see also
\cite[Example~2.10]{Asa22}).
Condition (i) implies that the restriction of $f_0^2$ to $A\cap f_0^{-1}(B)$ and the restriction of $f_0^2$
to $D\cap f_0^{-1}(C)$ are uniformly contracting in the $x_c$-direction (recall \eqref{constant1}).
Condition (ii) implies the so-called ``overlapping condition''.
It is easy to check that (i) (ii) are compatible with the assumption $a_1a_2>1$
in Theorems~A and B.
The set
\[\Psi=\bigcap_{n=-\infty}^\infty f_0^{-2n}((A\cap f_0^{-1}(B)) \cup (D\cap f_0^{-1}(C)))\]
is a hyperbolic set for $f_0^2$ of index $1$, and it is a $d_{cu}$-unstable blender.
Let $Q$ denote the fixed saddle of $f_0$ in $C$.
  Then $f_0$ has a $C^1$-robust heterodimensional cycle
 associated to $\Psi\cup f_0(\Psi)$ and $\{Q\}$.
The hyperbolic set 
$\Psi\cup f_0(\Psi)$ for $f_0$ is disjoint from another hyperbolic set $\Gamma(f_0)$ for $f_0$.

\subsection*{Acknowledgments}
We thank anonymous referees for their useful comments, and Masayuki Asaoka for fruitful discussions.
This research was supported by the JSPS KAKENHI 19KK0067, 19K21835, 20H01811 and JPJSBP 120229913.

\end{document}